\documentclass[a4paper,12pt]{article}
\usepackage{mathtools}
\usepackage{amsfonts}
\usepackage{amsthm}
\usepackage{graphicx}
\usepackage{framed}
\usepackage{float}
\usepackage{upgreek}
\usepackage{enumerate}
\usepackage{MnSymbol}
\usepackage{amsbsy}
\usepackage{fullpage}
\usepackage{bbm}
\PassOptionsToPackage{hyphens}{url}
\usepackage{hyperref}
\theoremstyle{plain}
\newtheorem{thm}{Theorem}[section]
\newtheorem{lemma}[thm]{Lemma}
\newtheorem{cor}[thm]{Corollary}

\newtheorem*{propn}{Proposition}

\newtheorem{prop}[thm]{Proposition}

\theoremstyle{definition}
\newtheorem{defi}[thm]{Definition}
\newtheorem{rmk}[thm]{Remark}
\newtheorem{ex}[thm]{Example}

\theoremstyle{remark}

\begin{document}

\title{Synchronisation in Invertible Random Dynamical Systems on the Circle}
\author{Julian Newman}
\maketitle

\begin{abstract}
\noindent In this paper, we study geometric features of orientation-preserving random dynamical systems on the circle driven by memoryless noise that exhibit stable synchronisation: we consider crack points, invariant measures, and the link between synchronisation and compressibility of arcs; we also characterise stable synchronisation in additive-noise stochastic differential equations on the circle, in terms of ``subperiodicity'' of the vector field.
\end{abstract}

\section{Introduction}

It is well-known that any ``sufficiently noisy'' invertible random dynamical system on the circle driven by memoryless noise exhibits ``contraction of orbits'' or ``synchronisation'', in the sense that the distance between the trajectories of two given initial conditions almost surely converges to $0$ as time tends to $\infty$. In discrete time, we have the following:

\begin{propn}[{[Ant84], [Mal14]}]
Given a set $F \subset \mathrm{Homeo}^+(\mathbb{S}^1)$, equipped with the uniform topology, and a probability measure $\nu$ on $F$ with full support, if either:
\begin{enumerate}[\indent (a)]
\item there is no finite-order orientation-preserving homeomorphism other than $\mathrm{id}_{\mathbb{S}^1}$ that commutes with every $f \in F$, and for every $x \in \mathbb{S}^1$ and open $U \subset \mathbb{S}^1$, there exist $f_1,\ldots,f_n,\tilde{f}_1,\ldots,\tilde{f}_m \in F$ such that
\[ f_n \circ \ldots \circ f_1(x) \in U \hspace{4mm} \textrm{and} \hspace{4mm} x \in \tilde{f}_m \circ \ldots \circ \tilde{f}_1(U) \, ; \textrm{ or} \]
\item for every distinct $x,y \in \mathbb{S}^1$ there exist $f_1,\ldots,f_n \in F$ such that
\[ d( \, f_n \circ \ldots \circ f_1(x) \, , \, f_n \circ \ldots \circ f_1(y) \, ) \ < \ d(x,y) \]
and there does not exist $p \in \mathbb{S}^1$ such that for every $f \in F$, $f(p)=p$;
\end{enumerate}
then given any $x,y \in \mathbb{S}^1$, we have that for $\nu^{\otimes \mathbb{N}}$-almost all $(f_n)_{n \geq 1}$,
\[ d( \, f_n \circ \ldots \circ f_1(x) \, , \, f_n \circ \ldots \circ f_1(y) \, ) \ \to \ 0 \ \textrm{ as } n \to \infty. \]
\end{propn}

\noindent In case~(a), the result is due to [Ant84]; in case~(b), the result is due to [Mal14]. Moreover, it is shown in [Mal14] that the convergence occurs at an exponential rate.
\\ \\
Now by [Mal14, Theorem~A], we can add the following to the conclusion in the above proposition: \emph{given any $x \in \mathbb{S}^1$, we have that for $\nu^{\otimes \mathbb{N}}$-almost every $(f_n)_{n \geq 1}$, there exists a neighbourhood $U$ of $x$ such that}
\[ \mathrm{diam}( \, f_n \circ \ldots \circ f_1(U) \, ) \ \to \ 0 \ \textit{ as } n \to \infty. \]
(Again, the convergence is at an exponential rate.) This additional property implies physically that small unexpected perturbations to the evolution of the trajectories are unlikely to destroy the synchronisation described in the above proposition. Hence, we refer to synchronisation combined with this additional property as ``stable synchronisation''.
\\ \\
Synchronisation in continuous-time systems on the circle has been studied in [Crau02], and some specific examples in [Bax86]. Necessary and sufficient conditions for stable synchronisation in a more general context have been given in [New17].
\\ \\
The goal of this paper is to describe certain geometrical features of orientation-preserving random dynamical systems on the circle exhibiting stable synchronisation. Our results apply in both discrete and continuous time. In Section~2, we will introduce our setting. In Section~3, we will present a characterisation of stable synchronisation in terms of ``crack points''. On the basis of this, in Section~4 we will describe the ``invariant measures'' of systems exhibiting stable synchronisation. In Section~5, we will present a result linking contractibility for pairs of trajectories, compressibility for arcs, and stable synchronisation; we will then characterise stable synchronisation in additive-noise stochastic differential equations on the circle, in terms of ``subperiodicity'' of the vector field.

\section{Our setting}

Let $\mathbb{T}^+$ denote either $\mathbb{N} \cup \{0\}$ or $[0,\infty)$. Let $(\Omega,\mathcal{F},(\mathcal{F}_t)_{t \in \mathbb{T}^+},\mathbb{P})$ be a filtered probability space (the ``noise space''), and write $\mathcal{F}_+:=\sigma(\mathcal{F}_t:t \in \mathbb{T}^+)$. Let $(\theta^t)_{t \in \mathbb{T}^+}$ be a family of $\mathbb{P}$-preserving $(\mathcal{F},\mathcal{F})$-measurable functions $\theta^t:\Omega \to \Omega$ such that $\theta^0=\mathrm{id}_\Omega$ and $\theta^{s+t}=\theta^t \circ \theta^s$ for all $s,t \in \mathbb{T}^+$. Suppose moreover that:
\begin{enumerate}[\indent (i)]
\item $\theta^t$ is $(\mathcal{F}_{s+t},\mathcal{F}_s)$-measurable (i.e.~$\theta^{-t}\mathcal{F}_s \subset \mathcal{F}_{s+t}$) for all $s,t \in \mathbb{T}^+$;
\item for each $t \in \mathbb{T}^+$, $\mathcal{F}_t$ and $\theta^{-t}\mathcal{F}_+$ are independent $\sigma$-algebras according to $\mathbb{P}$ (i.e. $\mathbb{P}(E \cap \theta^{-t}(F)) \, = \, \mathbb{P}(E)\mathbb{P}(F)$ for all $E \in \mathcal{F}_t$ and $F \in \mathcal{F}_+$).
\end{enumerate}

\noindent (Here, we use the convention $\theta^{-t}(E):=(\theta^t)^{-1}(E)$.)
\\ \\
Heuristically, as suggested by (i), $\theta^t\omega$ represents a \emph{time-shift} of the noise realisation $\omega$ forward by time $t$. The fact that $\mathbb{P}$ is invariant under $(\theta^t)$ represents the assumption that the noise is strictly stationary, and condition~(ii) represents the assumption that the noise is memoryless. We emphasise that, whether we are considering a one-sided-time noise process or a two-sided-time noise process, $\mathcal{F}_t$ always represents the information available between time 0 and time $t$.

\begin{ex}[Gaussian white noise]
Following sections~A.2 and A.3 of [Arn98], an ``eternal'' one-dimensional Gaussian white noise process may be described according to the framework above as follows: Let $\Omega:=\{\omega \in C(\mathbb{R},\mathbb{R}) : \omega(0)=0\}$. For each $t \in [0,\infty)$, let $\mathcal{F}_t$ be the smallest $\sigma$-algebra on $\Omega$ with respect to which the projection $W_s:\omega \mapsto \omega(s)$ is measurable for every $s \in [0,t]$. Let $\mathcal{F}$ be the smallest $\sigma$-algebra on $\Omega$ with respect to which the projection $W_s:\omega \mapsto \omega(s)$ is measurable for every $s \in \mathbb{R}$. Let $\mathbb{P}$ be the \emph{Wiener measure} on $(\Omega,\mathcal{F})$---that is, $\mathbb{P}$ is the unique probability measure under which the stochastic processes $(W_t)_{t \geq 0}$ and $(W_{\!-t})_{t \geq 0}$ are independent Wiener processes. Finally, for each $\tau \geq 0$ and $s \in \mathbb{R}$, set $\theta^\tau\omega(s):=\omega(\tau+s)-\omega(\tau)$.
\end{ex}

\noindent Now let $\mathbb{S}^1$ be the unit circle, which we identify with $^{\mathbb{R}\!}/_{\!\mathbb{Z}\,}$ in the obvious manner, and let $l$ denote the Lebesgue measure on $\mathbb{S}^1$ (with $l(\mathbb{S}^1)=1$). Let $\pi:\mathbb{R} \to \mathbb{S}^1$ denote the natural projection, i.e.~$\pi(x)\,=x+\mathbb{Z} \, \in \, \mathbb{S}^1$; a \emph{lift} of a point $x \in \mathbb{S}^1$ is a point $x' \in \mathbb{R}$ such that $\pi(x')=x$, and a lift of a set $A \subset \mathbb{S}^1$ is a set $B \subset \mathbb{R}$ such that $\pi(B)=A$. Define the metric $d$ on $\mathbb{S}^1$ by
\[ d(x,y) \ = \ \min\{|x'-y'| : \, x' \textrm{ is a lift of } x, \, y' \textrm{ is a lift of } y \}. \]
\noindent Note that under this metric, for any connected $J \subset \mathbb{S}^1$,
\[ \mathrm{diam}\,J \ = \ \min\!\left(l(J),\tfrac{1}{2}\right)\!. \]
\noindent Let $\,\varphi \!\! = \!\! \left(\varphi(t,\omega)\right)_{t \in \mathbb{T}^+ \! , \, \omega \in \Omega}\,$ be a $(\mathbb{T}^+ \! \times \Omega)$-indexed family of orientation-preserving homeomorphisms $\varphi(t,\omega):\mathbb{S}^1 \to \mathbb{S}^1$ such that:
\begin{enumerate}[\indent (a)]
\item the map $(\omega,x) \mapsto \varphi(t,\omega)x$ is $(\mathcal{F}_t \otimes \mathcal{B}(\mathbb{S}^1),\mathcal{B}(\mathbb{S}^1))$-measurable for each $t \in \mathbb{T}^+$;
\item $\varphi(0,\omega) \, = \, \mathrm{id}_{\mathbb{S}^1}$ for all $\omega \in \Omega$;
\item $\varphi(s+t,\omega) \, = \, \varphi(t,\theta^s\omega) \circ \varphi(s,\omega)\,$ for all $s,t \in \mathbb{T}^+$ and $\omega \in \Omega$;
\item for any decreasing sequence $(t_n)$ in $\mathbb{T}^+$ converging to a time $t$, and any sequence $(x_n)$ in $\mathbb{S}^1$ converging to a point $x$, $\,\varphi(t_n,\omega)x_n \to \varphi(t,\omega)x\,$ as $n \to \infty$ for all $\omega \in \Omega$;
\item there exists a function $\varphi_-:\mathbb{T}^+ \times \Omega \times \mathbb{S}^1 \to \mathbb{S}^1$ such that for any strictly increasing sequence $(t_n)$ in $\mathbb{T}^+$ converging to a time $t$, and any sequence $(x_n)$ in $\mathbb{S}^1$ converging to a point $x$, $\,\varphi(t_n,\omega)x_n \to \varphi_-(t,\omega,x)\,$ as $n \to \infty$ for all $\omega \in \Omega$.
\end{enumerate}

\noindent We refer to $\varphi$ as a \emph{random dynamical system} (RDS) on $\mathbb{S}^1$; more specifically, since $\varphi(t,\omega)$ is a homeomorphism for all $t$ and $\omega$, we refer to $\varphi$ as an \emph{invertible RDS}. Conditions~(d) and (e) constitute the ``c\`{a}dl\`{a}g'' property, with (d) being right-continuity and (e) being left limits.\footnote{The left-limits property is included simply to ensure that ``asymptotic stability'' (defined as the existence of a neighbourhood of the initial condition that contracts in diameter to $0$ under the flow) implies stability in the sense of Lyapunov.} It is not hard to show that property~(d) implies the following:
\begin{enumerate}[\indent (d')]
\item for any decreasing sequence $(t_n)$ in $\mathbb{T}^+$ converging to a time $t$, and any sequence $(x_n)$ in $\mathbb{S}^1$ converging to a point $x$, $\,\varphi(t_n,\omega)^{-1}(x_n) \to \varphi^{-1}(t,\omega)(x)\,$ as $n \to \infty$ for all $\omega \in \Omega$.
\end{enumerate}

\noindent We will say that $\varphi$ is a \emph{continuous RDS} if for all $\omega \in \Omega$ the map $(t,x) \mapsto \varphi(t,\omega)x$ is jointly continuous. (In this case, $(t,x) \mapsto \varphi(t,\omega)^{-1}(x)$ is also jointly continuous for all $\omega$.)

\begin{defi}
We say that $\varphi$ is \emph{synchronising} if for all $x,y \in \mathbb{S}^1$,
\[ \mathbb{P}( \, \omega \, : \, d(\varphi(t,\omega)x,\varphi(t,\omega)y) \to 0 \textrm{ as } t \to \infty \, ) \ = \ 1. \]
\end{defi}

\begin{defi}
We say that $\varphi$ is \emph{everywhere locally stable} if for all $x \in \mathbb{S}^1$,
\[ \mathbb{P}( \, \omega \, : \, \exists \, \textrm{open } U \! \ni x \, \textrm{ s.t.~} l(\varphi(t,\omega)U) \to 0 \textrm{ as } t \to \infty \, ) \ = \ 1 \, ; \]
\noindent and we say that $\varphi$ is \emph{stably synchronising} if $\varphi$ is both synchronising and everywhere locally stable.
\end{defi}

\noindent An example of a system that is synchronising but not stably synchronising is the following: Within a \emph{deterministic} setting (i.e.~taking $\Omega$ to be just a singleton $\{\omega\}$), let $f:\mathbb{S}^1 \to \mathbb{S}^1$ be an orientation-preserving homeomorphism with a unique fixed point $p$, and (working in discrete time) take $\varphi(n,\omega):=f^n$. For every $x \in \mathbb{S}^1$, $f^n(x) \to p$ as $n \to \infty$. Hence in particular, for all $x,y \in \mathbb{S}^1$, $d(f^n(x),f^n(y)) \to 0$ as $n \to \infty$; and yet for every neighbourhood $U$ of $p$, $l(f^n(U))$ tends to 1 rather than to 0 as $n \to \infty$.

\section{Crack points}

Observe that for any $\omega \in \Omega$, the binary relation $\sim_\omega$ on $\mathbb{S}^1$ defined by
\[ x \sim_\omega y \hspace{3mm} \Longleftrightarrow \hspace{3mm} d(\varphi(t,\omega)x,\varphi(t,\omega)y) \to 0 \textrm{ as } t \to \infty \]
\noindent is an equivalence relation. It is easy to show (by considering only rational times) that the set $\,\{(x,y,\omega) \in \mathbb{S}^1 \times \mathbb{S}^1 \times \Omega \, : \, x \sim_\omega y \}\,$ is a $(\mathcal{B}(\mathbb{S}^1 \times \mathbb{S}^1) \otimes \mathcal{F}_+)$-measurable set.

\begin{defi}[c.f.~\textrm{[Kai93]}]
Given a point $r \in \mathbb{S}^1$ and a sample point $\omega \in \Omega$, we will say that $r$ is a \emph{crack point} of $\omega$ if the following equivalent statements hold:
\begin{itemize}
\item for every $A \subset \mathbb{S}^1$ with $r \nin \bar{A}$, $\mathrm{diam}(\varphi(t,\omega)A) \to 0$ as $t \to \infty$;
\item for every closed $G \subset \mathbb{S}^1$ with $r \nin G$, $l(\varphi(t,\omega)G) \to 0$ as $t \to \infty$;
\item for every open $U \subset \mathbb{S}^1$ with $r \in U$, $l(\varphi(t,\omega)U) \to 1$ as $t \to \infty$.
\end{itemize}
\end{defi}

\noindent Obviously, any sample point admits at most one crack point. If a sample point $\omega$ admits a crack point, then we will say that $\omega$ is \emph{contractive}.
\\ \\
Now if a sample point $\omega$ admits a crack point $r$, then either (a)~the equivalence relation $\sim_\omega$ has two equivalence classes, namely $\{r\}$ and $\mathbb{S}^1 \setminus \{r\}$, or (b)~the equivalence relation $\sim_\omega$ has one equivalence class (the whole of $\mathbb{S}^1$). In case~(a), we say that $r$ is a \emph{repulsive} crack point of $\omega$.

\begin{defi}
Let $\Omega_c \subset \Omega$ be the set of contractive sample points, and let $\tilde{r}:\Omega_c \to \mathbb{S}^1$ denote the function sending a contractive sample point $\omega$ onto its crack point $\tilde{r}(\omega)$.
\end{defi}

\begin{lemma} \label{crrfp}
$\Omega_c$ is $\mathcal{F}_+$-measurable, and $\tilde{r}:\Omega_c \to \mathbb{S}^1$ is measurable with respect to the $\sigma$-algebra $\mathcal{F}_c$ of $\mathcal{F}_+$-measurable subsets of $\Omega_c$. For each $t \in \mathbb{T}^+$, $\theta^{-t}(\Omega_c) = \Omega_c$ and $\tilde{r}(\theta^t\omega)=\varphi(t,\omega)\tilde{r}(\omega)$ for all $\omega \in \Omega_c$.
\end{lemma}

\begin{proof}
Let $R$ be a countable dense subset of $\mathbb{S}^1$. For any connected $J \subset \mathbb{S}^1$, it is clear (by considering rational times) that
\begin{equation} \label{meas}
\{ \omega \in \Omega \, : \, l(\varphi(t,\omega)J) \to 0 \textrm{ as } t \to \infty \} \; \in \, \mathcal{F}_+.
\end{equation}
\noindent So then, in order to show that $\Omega_c \in \mathcal{F}_+$, it suffices to prove the following statement: a sample point $\omega \in \Omega$ is contractive if and only if for every $n \in \mathbb{N}$ there is a connected open set $U_n \subset \mathbb{S}^1$ with endpoints in $R$ such that $1-\frac{1}{n} < l(U_n) < 1$ and $l(\varphi(t,\omega)U_n) \to 0$ as $t \to \infty$. Now the ``only if'' direction is obvious. For the ``if'' direction: suppose that for every $n \in \mathbb{N}$ there exists a connected open set $U_n \subset \mathbb{S}^1$ with endpoints in $R$ such that $1-\frac{1}{n} < l(U_n) < 1$ and $l(\varphi(t,\omega)U_n) \to 0$ as $t \to \infty$; and let $U:=\bigcup_{n=1}^\infty U_n$. Since $U_n$ is connected for all $n$ and $l(U_n) \to 1$ as $n \to \infty$, we clearly have that either $U=\mathbb{S}^1$ or $\mathbb{S}^1 \setminus \{U\}$ is a singleton. Now suppose, for a contradiction, that $U=\mathbb{S}^1$. Then, since $\mathbb{S}^1$ is compact, there is a finite subset $\{n_1,\ldots,n_k\}$ of $\mathbb{N}$ such that $\mathbb{S}^1=\bigcup_{i=1}^k U_{n_i}$; but since $l(\varphi(t,\omega)U_{n_i}) \to 0$ as $t \to \infty$ for each $i$, we then have that $l(\varphi(t,\omega)\mathbb{S}^1) \to 0$ as $t \to \infty$, which is absurd. So then, we must have that $\mathbb{S}^1 \setminus U$ is equal to a singleton $\{r\}$. We now show that $r$ is a crack point. Fix any closed $G \subset \mathbb{S}^1$ with $r \nin G$. Take $n$ such that $l(U_n)>1-d(r,G)$; then $G \subset U_n$ and so $l(\varphi(t,\omega)G) \to 0$ as $t \to \infty$. Hence $r$ is a crack point of $\omega$.
\\ \\
Thus we have shown that $\Omega_c$ is $\mathcal{F}_+$-measurable. Now for any $\omega \in \Omega_c$ and any non-empty closed connected $K \subset \mathbb{S}^1$, observe that $\tilde{r}(\omega) \in K$ if and only if for every closed connected $G \subset \mathbb{S}^1 \setminus K$ with $\partial G \subset R$, $l(\varphi(t,\omega)G) \to 0$ as $t \to \infty$. So by (\ref{meas}) and the countability of $R$, $\tilde{r}^{-1}(K) \in \mathcal{F}_c$ for every closed connected $K \subset \mathbb{S}^1$. Hence $\tilde{r}$ is $\mathcal{F}_c$-measurable.
\\ \\
Now fix any $t \in \mathbb{T}^+$ and $\omega \in \Omega$. First suppose that $\omega$ admits a crack point $r$:~then for any closed $G \not\ni \varphi(t,\omega)r$,
\[ \mathrm{diam}(\varphi(s,\theta^t\omega)G) \ = \ \mathrm{diam}(\varphi(s+t,\omega) \, (\varphi(t,\omega)^{-1}(G)) \, ) \to 0 \textrm{ as } s \to \infty \]
\noindent since $\varphi(t,\omega)^{-1}(G)$ is a closed set not containing $r$; so $\varphi(t,\omega)r$ is a crack point of $\theta^t\omega$. Now suppose that $\theta^t\omega$ admits a crack point $q$:~then for any closed $G \not\ni \varphi(t,\omega)^{-1}(q)$,
\[ \mathrm{diam}(\varphi(s+t,\omega)G) \ = \ \mathrm{diam}(\varphi(s,\theta^t\omega) \, (\varphi(t,\omega)G) \, ) \to 0 \textrm{ as } s \to \infty \]
\noindent since $\varphi(t,\omega)G$ is a closed set not containing $q$; so $\varphi(t,\omega)^{-1}(q)$ is a crack point of $\omega$. Thus we have proved that $\omega$ admits a crack point if and only if $\theta^t\omega$ admits a crack point, and that in this case, $\tilde{r}(\theta^t\omega)=\varphi(t,\omega)\tilde{r}(\omega)$.
\end{proof}

\begin{thm} \label{cr char}
$\mathbb{P}(\Omega_c)$ is equal to either $0$ or $1$. In the case that $\mathbb{P}(\Omega_c)=1$, either:
\begin{enumerate}[\indent (a)]
\item for every $x \in \mathbb{S}^1$, $\,\mathbb{P}(\omega \in \Omega_c \, : \, \tilde{r}(\omega) = x) \ = \ 0$; or
\item there exists a deterministic fixed point $p \in \mathbb{S}^1$ such that $\,\mathbb{P}(\omega \in \Omega_c \, : \, \tilde{r}(\omega) = p) \ = \ 1$.
\end{enumerate}
\noindent $\varphi$ is stably synchronising if and only if $\mathbb{P}(\Omega_c)=1$ and case~(a) holds. In this case, we also have that for $\mathbb{P}$-almost every $\omega \in \Omega_c$, $\tilde{r}(\omega)$ is a repulsive crack point of $\omega$.
\end{thm}

\begin{rmk}
In the case that there is no deterministic fixed point, the fact that stable synchronisation implies $\mathbb{P}(\Omega_c)=1$ can also be derived using results from [Mal14].
\end{rmk}

\subsection*{Proof of Theorem~\ref{cr char}}

\begin{lemma} \label{empf}
The measure-preserving flow $(\Omega,\mathcal{F}_+,\mathbb{P}|_{\mathcal{F}_+},(\theta^t)_{t \in \mathbb{T}^+})$ is ergodic.
\end{lemma}

\noindent  For a proof, see e.g.~[New15, Corollary~133].

\begin{cor}
$\mathbb{P}(\Omega_c)$ is equal to either $0$ or $1$.
\end{cor}

\begin{proof}
Follows immediately from Lemmas~\ref{crrfp} and \ref{empf}.
\end{proof}

\noindent Now for each $x \in \mathbb{S}^1$ and $t \in \mathbb{T}^+$, define the probability measure $\bar{\varphi}_x^t$ on $\mathbb{S}^1$ by
\[ \bar{\varphi}_x^t(A) \ \, := \ \, \mathbb{P}( \, \omega \, : \, x \in \varphi(t,\omega)A \, ) \ = \ \mathbb{P}( \, \omega \, : \, x \in \varphi(t,\theta^s\omega)A \, ) \]
\noindent for all $A \in \mathcal{B}(\mathbb{S}^1)$ and any $s \in \mathbb{T}^+$. Given any $s,t \in \mathbb{T}^+$, observe that
\begin{itemize}
\item under the random map $\varphi(s,\omega)^{-1}\!:\mathbb{S}^1 \!\to \mathbb{S}^1$, the transition probability from a point $y$ to a set $A$ is precisely $\bar{\varphi}_y^s(A)$;
\item under the random map $\varphi(t,\theta^s\omega)^{-1}\!:\mathbb{S}^1 \!\to \mathbb{S}^1$, the transition probability from a point $x$ to a set $Y$ is precisely $\bar{\varphi}_x^t(Y)$;
\item under the random map $\varphi(s+t,\omega)^{-1}\!:\mathbb{S}^1 \!\to \mathbb{S}^1$, the transition probability from a point $x$ to a set $A$ is precisely $\bar{\varphi}_x^{s+t}(A)$.
\end{itemize}
\noindent Therefore, since the $\sigma$-algebras $\mathcal{F}_s$ and $\theta^{-s}\mathcal{F}_t$ are independent, the Chapman-Kolmogorov equation
\[ \bar{\varphi}_x^{s+t}(A) \ = \ \int_{\mathbb{S}^1} \bar{\varphi}_y^s(A) \; \bar{\varphi}_x^t(dy) \]
\noindent is satisfied for any $x \in \mathbb{S}^1$, $A \in \mathcal{B}(\mathbb{S}^1)$ and $s,t \in \mathbb{T}^+$. Moreover, since the map $(t,x) \mapsto \varphi(t,\omega)^{-1}(x)$ is jointly continuous in $x$ and right-continuous in $t$ for every $\omega \in \Omega$, the dominated convergence theorem gives that the map $(t,x) \mapsto \bar{\varphi}_x^t$ is (with respect to the topology of weak convergence) jointly continuous in $x$ and right-continuous in $t$.
\\ \\
We will say that a probability measure $\rho$ on $\mathbb{S}^1$ is \emph{reverse-stationary (with respect to $\varphi$)} if for all $t \in \mathbb{T}^+$ and $A \in \mathcal{B}(\mathbb{S}^1)$,
\[ \rho(A) \ = \ \int_\Omega \rho(\varphi(t,\omega)A) \, \mathbb{P}(d\omega). \hspace{0.2mm}\footnotemark \]
\footnotetext{If $\theta^\tau$ is a measurable automorphism of $(\Omega,\mathcal{F})$ for all $\tau \in \mathbb{T}^+$, then one can naturally define $\varphi(t,\omega)$ for negative $t$ by $\varphi(t,\omega):=\varphi(|t|,\theta^t\omega)^{-1}$. However, we emphasise that even in this case, in the definition of reverse-stationarity we must restrict to nonnegative $t$.}Note that for any $s \in \mathbb{T}^+$, since $\mathbb{P}$ is $\theta^s$-invariant this is equivalent to saying that for all $t \in \mathbb{T}^+$ and $A \in \mathcal{B}(\mathbb{S}^1)$,
\[ \rho(A) \ = \ \int_\Omega \rho(\varphi(t,\theta^s\omega)A) \, \mathbb{P}(d\omega). \]
Note also that $\rho$ is reverse-stationary if and only if $\rho$ is a stationary measure of the family of transition probabilities $(\bar{\varphi}_x^t)_{x \in \mathbb{S}^1 \! , \, t \in \mathbb{T}^+}$, i.e.
\[ \rho(A) \ = \ \int_{\mathbb{S}^1} \bar{\varphi}_x^t(A) \, \rho(dx) \]
for all $t \in \mathbb{T}^+$ and $A \in \mathcal{B}(\mathbb{S}^1)$. Therefore, by the Krylov-Bogolyubov theorem (e.g.~[New15, Theorem~114] or [Kif86, Lemma~5.2.1]), there must exist at least one reverse-stationary probability measure.

\begin{defi}
We say that a point $p \in \mathbb{S}^1$ is a \emph{deterministic fixed point} (\emph{of $\varphi$}) if $\mathbb{P}$-almost every $\omega \in \Omega$ has the property that for all $t \in \mathbb{T}^+$, $\varphi(t,\omega)p=p$.
\end{defi}

\begin{defi}
We say that a set $A \subset \mathbb{S}^1$ is \emph{forward-invariant} (\emph{under $\varphi$}) if $\mathbb{P}$-almost every $\omega \in \Omega$ has the property that for all $t \in \mathbb{T}^+$, $\varphi(t,\omega)A \subset A$.
\end{defi}

\noindent Note that a \emph{finite} set $P \subset \mathbb{S}^1$, the following statements are equivalent:
\begin{itemize}
\item $P$ is forward-invariant;
\item $\mathbb{S}^1 \setminus P$ is forward-invariant;
\item for each $t \in \mathbb{T}^+$,
\[ \mathbb{P}( \, \omega \, : \, \varphi(t,\omega)P = P ) \ = \ 1. \]
\end{itemize}
Now we say that a probability measure $\rho$ on $\mathbb{S}^1$ is \emph{atomless} if for all $x \in \mathbb{S}^1$, $\rho(\{x\})=0$.

\begin{lemma} \label{atomic}
Let $\rho$ be a probability measure that is ergodic with respect to the family of transition probabilities $(\bar{\varphi}_x^t)_{x \in \mathbb{S}^1 \! , \, t \in \mathbb{T}^+}$. Then either $\rho$ is atomless, or $\rho \, = \, \frac{1}{|P|}\sum_{x \in P} \delta_x$ for some finite forward-invariant set $P \subset \mathbb{S}^1$.
\end{lemma}

\begin{proof}
Suppose that $\rho$ is not atomless. Let $m:=\max\{ \rho(\{x\}) : x \in \mathbb{S}^1 \}$ and let $P:=\{x \in \mathbb{S}^1 : \rho(\{x\})=m\}$. For any $t \in \mathbb{T}^+$ and $\omega \in \Omega$, if $P \neq \varphi(t,\omega)P$ then $\rho(\varphi(t,\omega)P)<\rho(P)$; so since $\rho$ is reverse-stationary, we have that for each $t \in \mathbb{T}^+$,
\[ \mathbb{P}( \, \omega \, : \, P = \varphi(t,\omega)P ) \ = \ 1, \]
i.e.\ $P$ is forward-invariant. Note that $\bar{\varphi}_x^t(P)=1$ for each $x \in P$ and $t \in \mathbb{T}^+$. So since $\rho$ is ergodic with respect to $(\bar{\varphi}_x^t)_{x \in \mathbb{S}^1 \! , \, t \in \mathbb{T}^+}$ and $\rho(P)>0$, it follows that $\rho(P)=1$.
\end{proof}

\begin{cor} \label{Dirac}
Let $\rho$ be as in Lemma~\ref{atomic}, and suppose moreover that $\varphi$ is synchronising. Then $\rho$ is either atomless or a Dirac mass at a deterministic fixed point.
\end{cor}

\begin{proof}
Since $\varphi$ is synchronising, any finite forward-invariant set $P$ must be a singleton. So the result is immediate.
\end{proof}

\begin{lemma} \label{Dirrfp}
Suppose we have an $\mathcal{F}_+$-measurable function $q:\Omega \to \mathbb{S}^1$ with the property that for each $t \in \mathbb{T}^+$, for $\mathbb{P}$-almost all $\omega \in \Omega$, $\varphi(t,\omega)q(\omega)=q(\theta^t\omega)$. Then $q_\ast\mathbb{P}$ is an ergodic measure of the family of transition probabilities $(\bar{\varphi}_x^t)_{x \in \mathbb{S}^1 \! , \, t \in \mathbb{T}^+}$, and is either atomless or a Dirac mass at a deterministic fixed point.
\end{lemma}

\begin{proof}
First we show that $q_\ast\mathbb{P}$ is stationary with respect to $(\bar{\varphi}_x^t)$ (i.e.~is reverse-stationary with respect to $\varphi$). Note that for each $t$, the map $\omega \mapsto q(\theta^t\omega)$ is $\theta^{-t}\mathcal{F}_+$-measurable. For any $t \in \mathbb{T}^+$ and $A \in \mathcal{B}(\mathbb{S}^1)$,
\begin{align*}
\int_{\mathbb{S}^1} \bar{\varphi}_x^t(A) \; q_\ast\mathbb{P}(dx) \ &= \ \int_{\mathbb{S}^1} \bar{\varphi}_x^t(A) \ (q \circ \theta^t)_\ast\mathbb{P}(dx) \hspace{3mm} \textrm{(since $\mathbb{P}$ is $\theta^t$-invariant)} \\
&= \ \int_\Omega \bar{\varphi}_{q(\theta^t\omega)}^t(A) \, \mathbb{P}(d\omega) \\
&= \ \int_\Omega \mathbb{P}( \hspace{0.2mm} \tilde{\omega} \, : \, \varphi(t,\tilde{\omega})^{-1}(q(\theta^t\omega)) \in A \hspace{0.2mm} ) \; \mathbb{P}(d\omega) \\
&= \ \mathbb{P}( \hspace{0.2mm} \omega \, : \, \varphi(t,\omega)^{-1}(q(\theta^t\omega)) \in A \hspace{0.2mm} ) \\
&\hspace{10mm} \textrm{since $\mathcal{F}_t$ and $\theta^{-t}\mathcal{F}_+$ are independent $\sigma$-algebras} \\
&= \ \mathbb{P}( \hspace{0.2mm} \omega \, : \, q(\omega) \in A \hspace{0.2mm} ) \\
&= \ q_\ast\mathbb{P}(A).
\end{align*}
\noindent Hence $q_\ast\mathbb{P}$ is stationary with respect to $(\bar{\varphi}_x^t)$. Now let $A \in \mathcal{B}(\mathbb{S}^1)$ be a set such that for each $t \in \mathbb{T}^+$, for $(q_\ast\mathbb{P})$-almost every $x \in A$, $\bar{\varphi}_x^t(A)=1$; to prove that $q_\ast\mathbb{P}$ is ergodic with respect to $(\bar{\varphi}_x^t)$, we need to show that $q_\ast\mathbb{P}(A) \in \{0,1\}$. Let $E:=q^{-1}(A) \in \mathcal{F}_+$, and for each $t \in \mathbb{T}^+$ let
\[ \tilde{E}_t \ := \ \{ \omega \, : \, \varphi(t,\omega)^{-1}(q(\theta^t\omega)) \in A \}. \]
\noindent Obviously $\mathbb{P}(E \triangle \tilde{E}_t)=0$ for each $t$. So then
\begin{align*}
\mathbb{P}( \, E \, \cap \, \theta^{-t}(E) \, ) \ &= \ \int_{\theta^{-t}(E)} \mathbb{P}(E|\theta^{-t}\mathcal{F}_+)(\omega) \; \mathbb{P}(d\omega) \\
&= \ \int_{\theta^{-t}(E)} \mathbb{P}(\tilde{E}_t|\theta^{-t}\mathcal{F}_+)(\omega) \; \mathbb{P}(d\omega) \\
&= \ \int_{\theta^{-t}(E)} \mathbb{P}( \hspace{0.2mm} \tilde{\omega} \, : \, \varphi(t,\tilde{\omega})^{-1}(q(\theta^t\omega)) \in A \hspace{0.2mm} ) \; \mathbb{P}(d\omega) \\
&\hspace{10mm} \textrm{since $\mathcal{F}_t$ and $\theta^{-t}\mathcal{F}_+$ are independent $\sigma$-algebras} \\
&= \ \int_{\theta^{-t}(E)} \bar{\varphi}_{q(\theta^t\omega)}^t(A) \, \mathbb{P}(d\omega) \\
&= \ \int_A \bar{\varphi}_x^t(A) \; q_\ast\mathbb{P}(dx) \\
&= \ q_\ast\mathbb{P}(A) \\
&= \ \mathbb{P}(E).
\end{align*}
\noindent Hence $\mathbb{P}(E \setminus \theta^{-t}(E))=0$ for each $t$. Therefore, since $E \in \mathcal{F}_+$, Lemma~\ref{empf} gives that $\mathbb{P}(E) \in \{0,1\}$. So $q_\ast\mathbb{P}(A) \in \{0,1\}$, as required.
\\ \\
Thus we have shown that $q_\ast\mathbb{P}$ is ergodic with respect to $(\bar{\varphi}_x^t)$. Now suppose that $q_\ast\mathbb{P}$ is not atomless. Then by Lemma~\ref{atomic}, there is a finite forward-invariant set $P \subset \mathbb{S}^1$ such that $q_\ast\mathbb{P} \, = \, \frac{1}{|P|}\sum_{x \in P} \delta_x$. Fix an arbitrary $x \in P$. Let $E:=q^{-1}(\{x\}) \in \mathcal{F}_+$, and for each $t \in \mathbb{T}^+$ let
\[ \tilde{E}_t \ := \ \{ \omega \, : \, q(\theta^t\omega) = \varphi(t,\omega)x \}. \]
\noindent Once again, $\mathbb{P}(E \triangle \tilde{E}_t)=0$ for each $t$. For each $n \in \mathbb{N}$, for any $F \in \mathcal{F}_n$, we have that
\begin{align*}
\mathbb{P}(E \cap F) \ &= \ \int_F \mathbb{P}(E|\mathcal{F}_n)(\omega) \; \mathbb{P}(d\omega) \\
&= \ \int_F \mathbb{P}(\tilde{E}_n|\mathcal{F}_n)(\omega) \; \mathbb{P}(d\omega) \\
&= \ \int_F \mathbb{P}( \hspace{0.2mm} \tilde{\omega} \, : \, q(\theta^n\tilde{\omega}) = \varphi(n,\omega)x \hspace{0.2mm} ) \; \mathbb{P}(d\omega) \\
&\hspace{10mm} \textrm{since $\mathcal{F}_n$ and $\theta^{-n}\mathcal{F}_+$ are independent $\sigma$-algebras} \\
&= \ \int_F q_\ast\mathbb{P} \hspace{0.2mm} ( \hspace{0.2mm} \{\varphi(n,\omega)x\} \hspace{0.2mm} ) \; \mathbb{P}(d\omega) \\
&= \ \int_F \, \tfrac{1}{|P|} \; \mathbb{P}(d\omega) \\
&= \ \tfrac{1}{|P|} \mathbb{P}(F) \\
&= \ \mathbb{P}(E)\mathbb{P}(F).
\end{align*}
\noindent So $E$ is independent of $\mathcal{F}_n$ for each $n \in \mathbb{N}$, and therefore $E$ is independent of $\mathcal{F}_+$. In particular, $E$ is independent of itself, and so $\mathbb{P}(E)=1$. Hence $|P|=1$, i.e.~$q_\ast\mathbb{P}$ is a Dirac mass at a deterministic fixed point.
\end{proof}

\begin{cor} \label{a or b}
If $\mathbb{P}(\Omega_c)=1$ then either case~(a) or case~(b) in the statement of Theorem~\ref{cr char} holds.
\end{cor}

\begin{proof}
Fix an arbitrary $k \in \mathbb{S}^1$, and define the function $q:\Omega \to \mathbb{S}^1$ by
\[ q(\omega) \ = \ \left\{ \begin{array}{c l} \tilde{r}(\omega) & \omega \in \Omega_c \\ k & \omega \in \Omega \setminus \Omega_c. \end{array} \right. \]
\noindent By Lemma~\ref{crrfp}, $q$ is an $\mathcal{F}_+$-measurable function. If $\mathbb{P}(\Omega_c)=1$ then by Lemma~\ref{crrfp}, for each $t \in \mathbb{T}^+$, for $\mathbb{P}$-almost all $\omega \in \Omega$, $q(\theta^t\omega)=\varphi(t,\omega)q(\omega)$. Hence Lemma~\ref{Dirrfp} gives the desired result.
\end{proof}

\begin{lemma} \label{atomless}
If $\varphi$ is stably synchronising then $\varphi$ admits at least one atomless reverse-stationary probability measure.
\end{lemma}

\begin{proof}
Suppose $\varphi$ is stably synchronising. First suppose that $\varphi$ does \emph{not} have a deterministic fixed point. We know that there exists at least one probability measure $\rho$ that is ergodic with respect to $(\bar{\varphi}_x^t)$; by Corollary~\ref{Dirac}, such a probability measure must be atomless. So now suppose that $\varphi$ \emph{does} have a deterministic fixed point $p$. Let $p' \in \mathbb{R}$ be a lift of $p$, and for each $v \in [0,1]$, let $J_v:=\pi([p',p'+v])$. Define the function $h:\Omega \to [0,1]$ by
\[ h(\omega) \ = \ \sup\{ v \in [0,1) \, : \, l(\varphi(t,\omega)J_v) \to 0 \textrm{ as } t \to \infty \}. \]
\noindent For any $c \in [0,1)$ and $\omega \in \Omega$, $h(\omega) > c$ if and only if there exists $v \in (c,1) \cap \mathbb{Q}$ such that $l(\varphi(t,\omega)J_v) \to 0$ as $t \to \infty$. Hence $h$ is $\mathcal{F}_+$-measurable. Now we know that for $\mathbb{P}$-almost every $\omega \in \Omega$ there exists an open neighbourhood $U$ of $p$ such that $l(\varphi(t,\omega)U) \to 0$ as $t \to \infty$. Hence $h(\omega) \in (0,1)$ for $\mathbb{P}$-almost all $\omega \in \Omega$.
\\ \\
Now define the function $q:\Omega \to \mathbb{S}^1$ by
\[ q(\omega) \ = \ \pi(p' + h(\omega)). \]
\noindent Since $h$ is $\mathcal{F}_+$-measurable, $q$ is $\mathcal{F}_+$-measurable. Given any $t \in \mathbb{T}^+$ and $\omega \in \Omega$, we have that for all $v \in [0,1)$,
\[ l(\varphi(s,\omega)J_v) \to 0 \textrm{ as } s \to \infty \hspace{4mm} \Longleftrightarrow \hspace{4mm} l(\varphi(s,\theta^t\omega) \, (\varphi(t,\omega)J_v) \, ) \to 0 \textrm{ as } s \to \infty \]
\noindent and therefore $q(\theta^t\omega)=\varphi(t,\omega)q(\omega)$. Hence, by Lemma~\ref{Dirrfp}, $q_\ast\mathbb{P}$ is ergodic with respect to $(\bar{\varphi}_x^t)$. Moreover, since $h(\omega) \in (0,1)$ for $\mathbb{P}$-almost all $\omega \in \Omega$, $q_\ast\mathbb{P}$ is not equal to $\delta_p$. Since $\varphi$ is synchronising, $\varphi$ cannot have more than one deterministic fixed point, and therefore by Corollary~\ref{Dirac} (or the second statement in Lemma~\ref{Dirrfp}) $q_\ast\mathbb{P}$ must be atomless. 
\end{proof}

\begin{lemma} \label{mart}
Suppose we have an atomless\footnote{The condition that $\rho$ is atomless can in fact be dropped, although the proof then becomes significantly longer, as it is harder to justify that the martingale $(h_t)_{t \in \mathbb{T}^+}$ almost surely has right-continuous sample paths. In any case, we will not need this for our purposes.} reverse-stationary probability measure $\rho$. Then for any connected $J \subset \mathbb{S}^1$, for $\mathbb{P}$-almost all $\omega \in \Omega$, $\rho(\varphi(t,\omega)J)$ is convergent as $t \to \infty$.
\end{lemma}

\noindent The main idea of the proof is the same as in [LeJ87, Lemme~1].

\begin{proof}
Fix a connected $J \subset \mathbb{S}^1$, and for each $t$ and $\omega$ let $h_t(\omega)=\rho(\varphi(t,\omega)J)$. Note that for each boundary point $x$ of $J$, the map $t \mapsto \varphi(t,\omega)x$ is right-continuous for all $\omega$. Hence, since $\rho$ is atomless, the map $t \mapsto h_t(\omega)$ is right-continuous for all $\omega$. So if we can show that $(h_t)_{t \in \mathbb{T}^+}$ is an $(\mathcal{F}_t)_{t \in \mathbb{T}^+}$-adapted martingale, then the martingale convergence theorem will give the desired result. Fix any $s,t \in \mathbb{T}^+$. We have that
\begin{align*}
\mathbb{E}[h_{s+t}|\mathcal{F}_s](\omega) \ &= \ \mathbb{E}[\,\tilde{\omega} \mapsto \rho(\varphi(s+t,\tilde{\omega})J)\,|\,\mathcal{F}_s\,](\omega) \\
&= \ \mathbb{E}[\,\tilde{\omega} \mapsto \rho( \hspace{0.2mm} \varphi(t,\theta^s\tilde{\omega})(\varphi(s,\tilde{\omega})J) \hspace{0.2mm} )\,|\,\mathcal{F}_s\,](\omega) \\
&= \ \mathbb{E}[\,\tilde{\omega} \mapsto \rho( \hspace{0.2mm} \varphi(t,\theta^s\tilde{\omega})(\varphi(s,\omega)J) \hspace{0.2mm} )\,] \\
& \hspace{10mm} \textrm{since $\mathcal{F}_s$ and $\theta^{-s}\mathcal{F}_t$ are independent $\sigma$-algebras} \\
&= \ \rho(\varphi(s,\omega)J) \\
& \hspace{10mm} \textrm{since $\rho$ is reverse-stationary} \\
&= \ h_s(\omega).
\end{align*}
\noindent So we are done.
\end{proof}

\begin{lemma} \label{1atomless}
If $\varphi$ is stably synchronising then $\varphi$ admits at least one atomless reverse-stationary probability measure.
\end{lemma}

\begin{proof}
Suppose $\varphi$ is stably synchronising. First suppose that $\varphi$ does \emph{not} have a deterministic fixed point. We know that there exists at least one probability measure $\rho$ that is ergodic with respect to $(\bar{\varphi}_x^t)$; by Corollary~\ref{Dirac}, such a probability measure must be atomless. So now suppose that $\varphi$ \emph{does} have a deterministic fixed point $p$. Let $p' \in \mathbb{R}$ be a lift of $p$, and for each $v \in [0,1]$, let $J_v:=\pi([p',p'+v])$. Define the function $h:\Omega \to [0,1]$ by
\[ h(\omega) \ = \ \sup\{ v \in [0,1) \, : \, l(\varphi(t,\omega)J_v) \to 0 \textrm{ as } t \to \infty \}. \]
\noindent For any $c \in [0,1)$ and $\omega \in \Omega$, $h(\omega) > c$ if and only if there exists $v \in (c,1) \cap \mathbb{Q}$ such that $l(\varphi(t,\omega)J_v) \to 0$ as $t \to \infty$. Hence $h$ is $\mathcal{F}_+$-measurable. Now we know that for $\mathbb{P}$-almost every $\omega \in \Omega$ there exists an open neighbourhood $U$ of $p$ such that $l(\varphi(t,\omega)U) \to 0$ as $t \to \infty$. Hence $h(\omega) \in (0,1)$ for $\mathbb{P}$-almost all $\omega \in \Omega$.
\\ \\
Now define the function $q:\Omega \to \mathbb{S}^1$ by
\[ q(\omega) \ = \ \pi(p' + h(\omega)). \]
\noindent Since $h$ is $\mathcal{F}_+$-measurable, $q$ is $\mathcal{F}_+$-measurable. Given any $t \in \mathbb{T}^+$ and $\omega \in \Omega$, we have that for all $v \in [0,1)$,
\[ l(\varphi(s,\omega)J_v) \to 0 \textrm{ as } s \to \infty \hspace{4mm} \Longleftrightarrow \hspace{4mm} l(\varphi(s,\theta^t\omega) \, (\varphi(t,\omega)J_v) \, ) \to 0 \textrm{ as } s \to \infty \]
\noindent and therefore $q(\theta^t\omega)=\varphi(t,\omega)q(\omega)$. Hence, by Lemma~\ref{Dirrfp}, $q_\ast\mathbb{P}$ is ergodic with respect to $(\bar{\varphi}_x^t)$. Moreover, since $h(\omega) \in (0,1)$ for $\mathbb{P}$-almost all $\omega \in \Omega$, $q_\ast\mathbb{P}$ is not equal to $\delta_p$. Since $\varphi$ is synchronising, $\varphi$ cannot have more than one deterministic fixed point, and therefore by Corollary~\ref{Dirac} (or the second statement in Lemma~\ref{Dirrfp}) $q_\ast\mathbb{P}$ must be atomless. 
\end{proof}

\begin{lemma} \label{cr1}
Suppose $\varphi$ is synchronising and admits an atomless reverse-stationary probability measure $\rho$. Then $\mathbb{P}(\Omega_c)=1$.
\end{lemma}

\begin{proof}
Given any connected $J \subset \mathbb{S}^1$ with $0<l(J)<1$, writing $\partial J \! =: \! \{x,y\}$, for $\mathbb{P}$-almost all $\omega \in \Omega$ we have the following: $d(\varphi(t,\omega)x,\varphi(t,\omega)y) \to 0$ as $t \to \infty$ (since $\varphi$ is synchronising), and $\rho(\varphi(t,\omega)J)$ is convergent as $t \to \infty$ (by Lemma~\ref{mart}); and therefore (since $\rho$ is atomless) $l(\varphi(t,\omega)J)$ converges to either 0 or 1 as $t \to \infty$. Now fix an arbitrary $k \in \mathbb{R}$, and for each $v \in [0,1]$, let $J_v:=\pi([k,k+v])$. Let $\Omega' \subset \Omega$ be a $\mathbb{P}$-full set such that for each $\omega \in \Omega'$ and $v \in [0,1] \cap \mathbb{Q}$, $l(\varphi(t,\omega)J_v)$ converges to either 0 or 1 as $t \to \infty$. For each $\omega \in \Omega'$, let
\begin{align*}
c(\omega) \ :=& \ \sup \{ v \in [0,1] : l(\varphi(t,\omega)J_v) \to 0 \textrm{ as } t \to \infty \} \\
=& \ \hspace{1.2mm} \inf \{ v \in [0,1] : l(\varphi(t,\omega)J_v) \to 1 \textrm{ as } t \to \infty \}.
\end{align*}
\noindent It is clear that for each $\omega \in \Omega'$, $\pi(k+c(\omega))$ is a crack point of $\omega$. So we are done.
\end{proof}

\noindent Combining Lemmas~\ref{1atomless} and \ref{cr1} gives that if $\varphi$ is stably synchronising then $\mathbb{P}(\Omega_c)=1$.

\begin{lemma}
Suppose $\mathbb{P}(\Omega_c)=1$. Then $\varphi$ is stably synchronising if and only if case~(a) in the statement of Theorem~\ref{cr char} holds.
\end{lemma}

\begin{proof}
For any $x,y \in \mathbb{S}^1$ and $\omega \in \Omega_c$, if $\tilde{r}(\omega) \neq x$ and $\tilde{r}(\omega) \neq y$ then $x \sim_\omega y$. Hence it is clear that in case~(a) in the statement of Theorem~\ref{cr char}, $\varphi$ is synchronising. For any $x \in \mathbb{S}^1$ and $\omega \in \Omega_c$, if $\tilde{r}(\omega) \neq x$ then there obviously exists a neighbourhood $U$ of $x$ such that $\mathrm{diam}(\varphi(t,\omega)U) \to 0$ as $t \to \infty$. Hence, in case~(a) in the statement of Theorem~\ref{cr char}, $\varphi$ is everywhere locally stable. Thus, in case~(a) in the statement of Theorem~\ref{cr char}, $\varphi$ is stably synchronising. Now if there exists $p \in \mathbb{S}^1$ such that $\mathbb{P}(\omega \in \Omega_c \, : \, \tilde{r}(\omega) = p) \, > \, 0$, then $\varphi$ is clearly not everywhere locally stable (since the point $p$ serves as a counterexample), and so $\varphi$ is not stably synchronising.
\end{proof}

\noindent Now we will say that a probability measure $\rho$ on $\mathbb{S}^1$ is \emph{forward-stationary} if for all $t \in \mathbb{T}^+$ and $A \in \mathcal{B}(\mathbb{S}^1)$,
\[ \rho(A) \ = \ \int_\Omega \rho(\varphi(t,\omega)^{-1}(A)) \, \mathbb{P}(d\omega). \]
Again, by the Krylov-Bogolyubov theorem, there is at least one forward-stationary probability measure.

\begin{lemma} \label{rep}
Let $q:\Omega \to \mathbb{S}^1$ be as in Lemma~\ref{Dirrfp}, and suppose that $q_\ast\mathbb{P}$ is atomless. Let $\rho$ be any forward-stationary probability measure. Then for $\mathbb{P}$-almost every $\omega \in \Omega$, $\rho(x \in \mathbb{S}^1 : x \sim_\omega q(\omega))=0$.
\end{lemma}

\begin{proof}
Define the functions
\begin{align*}
\Theta \, &: \, \Omega \times \mathbb{S}^1 \, \to \, \Omega \times \mathbb{S}^1 \\
\Theta_2 \, &: \, \Omega \times \mathbb{S}^1 \times \mathbb{S}^1 \, \to \, \Omega \times \mathbb{S}^1 \times \mathbb{S}^1
\end{align*}
\noindent \vspace{-3mm} by
\begin{align*}
\Theta(\omega,x) \ &= \ (\theta^1\omega,\varphi(1,\omega)x) \\
\Theta_2(\omega,x,y) \ &= \  (\theta^1\omega,\varphi(1,\omega)x,\varphi(1,\omega)y).
\end{align*}
\noindent It is not hard to show (e.g.~as in [New15, Theorem~143(i)] or [Kif86, Lemma~1.2.3]) that $\Theta$ is a measure-preserving transformation of $\left(\Omega \times \mathbb{S}^1 \, , \, \mathcal{F}_+ \otimes \mathcal{B}(\mathbb{S}^1) \, , \, \mathbb{P}|_{\mathcal{F}_+} \otimes \rho\right)$. Now define the probability measure $\mathfrak{p}$ on the measurable space $\left(\Omega \times \mathbb{S}^1 \times \mathbb{S}^1 \, , \, \mathcal{F}_+ \otimes \mathcal{B}(\mathbb{S}^1 \times \mathbb{S}^1)\right)$ by
\[ \mathfrak{p}(A) \ := \ \mathbb{P} \otimes \rho( \, (\omega,x) \in \Omega \times \mathbb{S}^1 \, : \, (\omega,x,q(\omega)) \in A \, ). \]
\noindent For any $A \in \mathcal{F}_+ \otimes \mathcal{B}(\mathbb{S}^1 \times \mathbb{S}^1)$, since $q$ is $\mathcal{F}_+$-measurable, the set $\{(\omega,x) : (\omega,x,q(\omega)) \in A\}$ is $(\mathcal{F}_+ \otimes \mathcal{B}(\mathbb{S}^1))$-measurable. With this, we have
\begin{align*}
\mathfrak{p}(\Theta_2^{-1}(A)) \ &= \ \mathbb{P} \otimes \rho( \, (\omega,x) \in \Omega \times \mathbb{S}^1 \, : \, (\theta^1\omega,\varphi(1,\omega)x,\varphi(1,\omega)q(\omega)) \in A \, ) \\
&= \ \mathbb{P} \otimes \rho( \, (\omega,x) \in \Omega \times \mathbb{S}^1 \, : \, (\theta^1\omega,\varphi(1,\omega)x,q(\theta^1\omega)) \in A \, ) \\
&= \ \mathbb{P} \otimes \rho( \, \Theta^{-1\!}\{(\omega,x) \in \Omega \times \mathbb{S}^1 \, : \, (\omega,x,q(\omega)) \in A\} \, ) \\
&= \ \mathbb{P} \otimes \rho( \, (\omega,x) \in \Omega \times \mathbb{S}^1 \, : \, (\omega,x,q(\omega)) \in A \, ) \\
&= \ \mathfrak{p}(A).
\end{align*}
\noindent So $\mathfrak{p}$ is $\Theta_2$-invariant. Now let $\Delta:=\{(x,x):x \in \mathbb{S}^1\}$. Since $q_\ast\mathbb{P}$ is atomless, we have that
\begin{align*}
\mathfrak{p}(\Omega \times \Delta) \ &= \ \mathbb{P} \otimes \rho( \, (\omega,x) \in \Omega \times \mathbb{S}^1 \, : \, q(\omega)=x \, ) \\
&= \ \int_{\mathbb{S}^1} \mathbb{P}(\omega \in \Omega : q(\omega)=x) \, \rho(dx) \\
&= \ 0.
\end{align*}
\noindent So letting $U_\varepsilon := \{(x,y) \in \mathbb{S}^1 \times \mathbb{S}^1 : d(x,y) < \varepsilon\}$ for each $\varepsilon>0$, we have that $\mathfrak{p}(\Omega \times U_\varepsilon) \to 0$ as $\varepsilon \to 0$. Hence the set
\begin{align*}
K \ :=& \ \{ \, (\omega,x,y) \in \Omega \times \mathbb{S}^1 \times \mathbb{S}^1 \, : \, d(\varphi(n,\omega)x,\varphi(n,\omega)y) \to 0 \textrm{ as } n \to \infty \, \} \\
=& \ \bigcap_{n=1}^\infty \, \bigcup_{i=1}^\infty \, \bigcap_{j=i}^\infty \; \Theta_2^{-j}(U_{\frac{1}{n}})
\end{align*}
\noindent is a $\mathfrak{p}$-null set. Hence the set
\[ L \ := \ \{ \, (\omega,x) \in \Omega \times \mathbb{S}^1 \, : \, d(\varphi(n,\omega)x,\varphi(n,\omega)q(\omega)) \to 0 \textrm{ as } n \to \infty \, \} \]
\noindent is a $(\mathbb{P} \otimes \rho)$-null set. So (with Fubini's theorem) we are done.
\end{proof}

\begin{cor}
If $\varphi$ is stably synchronising then for $\mathbb{P}$-almost all $\omega \in \Omega_c$, $\tilde{r}(\omega)$ is a repulsive crack point of $\omega$.
\end{cor}

\begin{proof}
Taking $q$ to be as in the proof of Corollary~\ref{a or b}, the result follows immediately from Lemma~\ref{rep} and the existence of a forward-stationary probability measure.
\end{proof}

\section{Invariant measures}

We will describe the set of invariant random probability measures for $\varphi$ in the case that $\mathbb{P}(\Omega_c)=1$ (and in particular, in the case that $\varphi$ is stably synchronising). We begin with motivation from the deterministic setting:
\\ \\
An \emph{autonomous flow of the circle} is a $\mathbb{T}^+$-indexed family $(f^t)_{t \in \mathbb{T}^+}$ of orientation-preserving homeomorphisms $f^t:\mathbb{S}^1 \to \mathbb{S}^1$ such that $f^{s+t}=f^t \circ f^s$ for all $s,t \in \mathbb{T}^+$ and $f^0$ is the identity function. Given an autonomous flow of the circle $(f^t)$, we will say that a point $p \in \mathbb{S}^1$ is a \emph{fixed point} if $f^t(p)=p$ for all $t \in \mathbb{T}^+$. Note that if $\Omega$ is a singleton $\{\omega\}$ then $(\varphi(t,\omega))_{t \in \mathbb{T}^+}$ is an autonomous flow of the circle.
\\ \\
Let $\mathcal{M}_1$ be the set of probability measures on $\mathbb{S}^1$. Now, heuristically, the map $x \mapsto \delta_x$ serves as a natural way of identifying points in $\mathbb{S}^1$ with measures on $\mathbb{S}^1$. On the basis of this identification, given an autonomous flow of the circle $(f^t)$, for each $t \in \mathbb{T}^+$ the map $\rho \mapsto f^t_\ast\rho$ on $\mathcal{M}_1$ serves as a natural way of ``lifting'' $f^t$ from $\mathbb{S}^1$ to $\mathcal{M}_1$, since $f^t_\ast\delta_x=\delta_{f^t(x)}$ for all $x \in \mathbb{S}^1$. In particular, a point $p \in \mathbb{S}^1$ is a fixed point of $(f^t)$ if and only if $\delta_p$ is an invariant measure (and in fact, an ergodic invariant measure) of $(f^t)$.
\\ \\
We will say that an autonomous flow of the circle $(f^t)$ is \emph{simple} if there exist distinct points $r,a \in \mathbb{S}^1$ (respectively called the \emph{repeller} and the \emph{attractor} of $(f^t)$) such that $r$ is a fixed point and for all $x \in \mathbb{S}^1 \setminus \{r\}$, $f^t(x) \to a$ as $t \to \infty$. (It follows that $a$ is also a fixed point.) In this case, it is easy to show that the set of invariant probability measures for $(f^t)$ is given by $\{\lambda \delta_r + (1-\lambda) \delta_a : \lambda \in [0,1]\}$. We will refer to the pair $(a,r)$ as a \emph{global attractor-repeller pair} for $(f^t)$; the basis for this terminology that ``$r$ repels all points in $\mathbb{S}^1$ (other than itself) towards $a$''.
\\ \\
If $\Omega$ is a singleton $\{\omega\}$, the flow $(\varphi(t,\omega))_{t \in \mathbb{T}^+}$ is simple if and only if $\omega$ admits a repulsive crack point $r$, in which case $r$ is the repeller of $(\varphi(t,\omega))_{t \in \mathbb{T}^+}$. Obviously, in this case, $\varphi$ is not synchronising, since the repeller and attractor of $(\varphi(t,\omega))_{t \in \mathbb{T}^+}$ are distinct deterministic fixed points of $\varphi$. (Note that it is impossible for $\varphi$ to be stably synchronising when $\Omega$ is a singleton.)
\\ \\
We now go on to extend the notion of ``simplicity'' to the random setting, and show that if $\varphi$ is stably synchronising then $\varphi$ is simple. We will need a couple of assumptions.
\\ \\
\textbf{Assumption A:} For all $t \in \mathbb{T}^+$, $\theta^t$ is a measurable automorphism of the measurable space $(\Omega,\mathcal{F})$.
\\ \\
The heuristic interpretation of this assumption is that the underlying noise process has been going on since eternity past. We refer to $\mathcal{F}_+$ as the \emph{future $\sigma$-algebra}, and we define the \emph{past $\sigma$-algebra} by $\mathcal{F}_- \, := \, \sigma(\theta^t\mathcal{F}_t:t \in \mathbb{T}^+) \subset \mathcal{F}$. It is not hard to show that $\mathcal{F}_+$ and $\mathcal{F}_-$ are independent (according to $\mathbb{P}$).
\\ \\
Let $\tilde{\mathcal{F}}:=\sigma(\theta^t\mathcal{F}_s:s,t \in \mathbb{T}^+)$. Note that $\mathcal{F}_+$ and $\mathcal{F}_-$ are both sub-$\sigma$-algebras of $\tilde{\mathcal{F}}$, and that for all $t \in \mathbb{T}^+$, $\theta^t$ is a measurable automorphism of $(\Omega,\tilde{\mathcal{F}})$. Moreover, we have the following:\footnote{Lemma~5.1 is a generalisation of some of the content of the section ``Invariant and Tail $\sigma$-Algebra'' on p547 of [Arn98].}

\begin{lemma}
For all $t \in \mathbb{T}^+ \setminus \{0\}$, $\theta^t$ is an ergodic measure-preserving transformation of the probability space $(\Omega,\tilde{\mathcal{F}},\mathbb{P}|_{\tilde{\mathcal{F}}})$.
\end{lemma}

\begin{proof}
Fix $t \in \mathbb{T}^+ \setminus \{0\}$. For each $n \in \mathbb{Z}$, let $\mathcal{G}_n=\theta^{nt}\mathcal{F}_+$. (Note that $\mathcal{G}_n$ is increasing in $n$.) Let $\mathcal{G}_{-\infty}:=\bigcap_{n \in \mathbb{Z}} \mathcal{G}_n$, and observe that $\tilde{\mathcal{F}}=\sigma(\bigcup_{n \in \mathbb{Z}} \mathcal{G}_n)$. Let $E \in \tilde{\mathcal{F}}$ be a set with $\theta^{-t}(E)=E$, and let $g:\Omega \to [0,1]$ be a version of $\mathbb{P}(E|\mathcal{F}_+)$; so for every $n \in \mathbb{Z}$, $g \circ \theta^{nt}$ is a version of $\mathbb{P}(E|\mathcal{G}_{-n})$. By a version of the Kolmogorov 0-1 law (e.g.~[New15, Proposition~132]), $\mathcal{G}_{-\infty}$ is a $\mathbb{P}$-trivial $\sigma$-algebra, and so the constant map $\,\omega \mapsto \mathbb{P}(E)$ is a version of $\mathbb{P}(E|\mathcal{G}_{-\infty})$. Therefore, by L\'{e}vy's downward theorem ([Will91, Theorem~14.4]), $g \circ \theta^{nt}(\omega) \to \mathbb{P}(E)$ as $n \to \infty$ for $\mathbb{P}$-almost all $\omega \in \Omega$. But since $\theta^t$ is itself $\mathbb{P}$-preserving, it follows that for each $n \in \mathbb{Z}$, $g \circ \theta^{nt}(\omega) = \mathbb{P}(E)$ for $\mathbb{P}$-almost all $\omega \in \Omega$.  So the constant map $\,\omega \mapsto \mathbb{P}(E)$ is a version of $\mathbb{P}(E|\mathcal{G}_n)$ for each $n$, i.e.~$E$ is independent of $\mathcal{G}_n$ for each $n$. It follows that $E$ is independent of $\tilde{\mathcal{F}}$. In particular, $E$ is independent of itself, and so $\mathbb{P}(E) \in \{0,1\}$.
\end{proof}

\noindent Hence (by reducing our underlying probability space to $(\Omega,\tilde{\mathcal{F}},\mathbb{P}|_{\tilde{\mathcal{F}}})$ if necessary) we may add the following assumption without loss of generality:
\\ \\
\textbf{Assumption B:} $\mathbb{P}$ is an ergodic measure of the dynamical system $(\theta^t)_{t \in \mathbb{T}^+}$ on $(\Omega,\mathcal{F})$.
\\ \\
\textbf{Throughout Section~4, we will work with Assumptions~A and B.}
\\ \\
We equip $\mathcal{M}_1$ with the ``evaluation $\sigma$-algebra'' $\mathfrak{K}:=\sigma(\rho \mapsto \rho(A) : A \in \mathcal{B}(\mathbb{S}^1))$. So for any measurable space $(E,\mathcal{E})$, a function $f:E \to \mathcal{M}_1$ is measurable if and only if the map $\,\xi \mapsto f(\xi)(A)\,$ from $E$ to $[0,1]$ is measurable for all $A \in \mathcal{B}(\mathbb{S}^1)$. It is not hard to show that the map $x \mapsto \delta_x$ is a measurable embedding of $\mathbb{S}^1$ into $\mathcal{M}_1$---that is to say, a set $A \subset \mathbb{S}^1$ is $\mathcal{B}(\mathbb{S}^1)$-measurable if and only if the set $\{\delta_x : x \in A\}$ is $\mathfrak{K}$-measurable.\footnote{On the one hand, the map $x \mapsto \delta_x$ is obviously measurable, and so if $\{\delta_x : x \in A\}$ is measurable then $A$ is measurable. On the other hand, for any $A \in \mathcal{B}(\mathbb{S}^1)$, $\{\delta_x:x \in A\}$ is precisely the set of probability measures $\rho$ for which $\rho \otimes \rho(\{(x,x): x \in A\})=1$. Now the set $\{(x,x): x \in A\}=\Delta \cap (A \times A)$ is obviously measurable, and it is not hard to show that the map $\rho \mapsto \rho \otimes \rho$ is measurable (with respect to the respective evaluation $\sigma$-algebras); therefore $\{\delta_x:x \in A\}$ is measurable.}

\begin{defi}
A \emph{random probability measure} (\emph{on $\mathbb{S}^1$}) is an $\Omega$-indexed family $(\mu_\omega)_{\omega \in \Omega}$ of probability measures $\mu_\omega$ on $\mathbb{S}^1$ such that the map $\omega \mapsto \mu_\omega$ from $\Omega$ to $\mathcal{M}_1$ is measurable (i.e.~such that for each $A \in \mathcal{B}(\mathbb{S}^1)$, the map $\omega \mapsto \mu_\omega(A)$ from $\Omega$ to $[0,1]$ is measurable). We say that two random probability measures $(\mu_\omega)_{\omega \in \Omega}$ and $(\mu'_\omega)_{\omega \in \Omega}$ are \emph{equivalent} if $\mu_\omega=\mu'_\omega$ for $\mathbb{P}$-almost all $\omega \in \Omega$.
\end{defi}

\begin{defi}
We will say that a probability measure $\mu$ on the measurable space $(\Omega \times \mathbb{S}^1,\mathcal{F} \otimes \mathcal{B}(\mathbb{S}^1))$ is \emph{compatible} if $\mu(E \times \mathbb{S}^1)=\mathbb{P}(E)$ for all $E \in \mathcal{F}$. We write $\mathcal{M}_1^\mathbb{P}$ for the set of compatible probability measures.
\end{defi}

\noindent The \emph{disintegration theorem} (e.g.~[Crau02a, Proposition~3.6]) states that for every compatible probability measure $\mu$ there exists a random probability measure $(\mu_\omega)_{\omega \in \Omega}$ (unique up to equivalence) such that
\[ \mu(A) \ = \ \int_\Omega \mu_\omega(A_\omega) \, \mathbb{P}(d\omega) \]
\noindent for all $A \in \mathcal{F} \otimes \mathcal{B}(\mathbb{S}^1)$, where $A_\omega$ denotes the $\omega$-section of $A$. The random probability measure $(\mu_\omega)$ is called \emph{a (version of the) disintegration of $\mu$}, and we will refer to $\mu$ as \emph{the integrated form of $(\mu_\omega)$}.

\begin{lemma} \label{ms}
Let $\mu^1$ and $\mu^2$ be compatible probability measures, with $(\mu^1_\omega)$ and $(\mu^2_\omega)$ being disintegrations of $\mu^1$ and $\mu^2$ respectively. If $\mu^1$ and $\mu^2$ are mutually singular then for $\mathbb{P}$-almost every $\omega \in \Omega$, $\mu^1_\omega$ and $\mu^2_\omega$ are mutually singular.
\end{lemma}

\begin{proof}
Suppose we have a set $A \in \mathcal{F} \otimes \mathcal{B}(\mathbb{S}^1)$ such that $\mu^1(A)=1$ and $\mu^2(A)=0$. Then it is clear that for $\mathbb{P}$-almost all $\omega \in \Omega$, $\mu^1_\omega(A_\omega)=1$ and $\mu^2_\omega(A_\omega)=0$. So we are done.
\end{proof}

\begin{defi}
For any measurable function $q:\Omega \to \mathbb{S}^1$, we define the probability measure $\boldsymbol{\delta}_q$ on $(\Omega \times \mathbb{S}^1,\mathcal{F} \otimes \mathcal{B}(\mathbb{S}^1))$ by $\boldsymbol{\delta}_q(A)=\mathbb{P}(\omega \in \Omega : (\omega,q(\omega)) \in A)$.
\end{defi}

\noindent Note that $\boldsymbol{\delta}_q$ is a compatible probability measure, and admits the disintegration $(\delta_{q(\omega)})_{\omega \in \Omega}$.
\\ \\
We will say that two measurable functions $q,q':\Omega \to \mathbb{S}^1$ are \emph{equivalent} if $q(\omega)=q'(\omega)$ for $\mathbb{P}$-almost all $\omega \in \Omega$. Let $L^0(\mathbb{P},\mathbb{S}^1)$ denote the set of equivalence classes of measurable functions from $\Omega$ to $\mathbb{S}^1$; and for any measurable $q:\Omega \to \mathbb{S}^1$, let $\hat{q} \in L^0(\mathbb{P},\mathbb{S}^1)$ denote the equivalence class represented by $q$. For any sub-$\sigma$-algebra $\mathcal{G}$ of $\mathcal{F}$, we say that an element $\hat{q}$ of $ L^0(\mathbb{P},\mathbb{S}^1)$ is $\mathcal{G}$-measurable if it admits a representative $q:\Omega \to \mathbb{S}^1$ that is $\mathcal{G}$-measurable.
\\ \\
Heuristically, $L^0(\mathbb{P},\mathbb{S}^1)$ can be viewed as the set of ``random points in $\mathbb{S}^1$'', identified up to equivalence; and by the disintegration theorem, $\mathcal{M}_1^\mathbb{P}$ can be regarded as the set of random probability measures on $\mathbb{S}^1$, identified up to equivalence. Note that the map $\hat{q} \mapsto \boldsymbol{\delta}_q$ from $L^0(\mathbb{P},\mathbb{S}^1)$ to $\mathcal{M}_1^\mathbb{P}$ is well-defined and injective. Thus, heuristically, this map serves as a natural way of identifying random points in $\mathbb{S}^1$ with random measures on $\mathbb{S}^1$.
\\ \\
Now for each $t \in \mathbb{T}^+$, define the map $\Phi^t:L^0(\mathbb{P},\mathbb{S}^1) \to L^0(\mathbb{P},\mathbb{S}^1)$ by
\[ r \in \Phi^t(\hat{q}) \hspace{3mm} \Longleftrightarrow \hspace{3mm} r(\omega) \, = \, \varphi(t,\theta^{-t}\omega)q(\theta^{-t}\omega) \textrm{ for $\mathbb{P}$-almost all $\omega \in \Omega$}. \]
\noindent (It is easy to show that since $\mathbb{P}$ is $\theta^t$-invariant, this is indeed a well-defined map.) One can easily check that $\Phi^0$ is the identity function and $\Phi^{s+t}=\Phi^t \circ \Phi^s$ for all $s,t \in \mathbb{T}^+$. If $\Omega$ is a singleton $\{\omega\}$ then we may identify $L^0(\mathbb{P},\mathbb{S}^1)$ with $\mathbb{S}^1$ in the obvious manner, in which case $\Phi^t$ is simply equal to $\varphi(t,\omega)$ for all $t$.

\begin{defi}
An element $p$ of $L^0(\mathbb{P},\mathbb{S}^1)$ is called a \emph{random fixed point} (\emph{of $\varphi$}) if $\Phi^t(p)=p$ for all $t \in \mathbb{T}^+$. In this case, for convenience, we will also refer to any representative of $p$ as a random fixed point.
\end{defi}

\noindent Note that a measurable function $q:\Omega \to \mathbb{S}^1$ is a random fixed point if and only if for each $t \in \mathbb{T}^+$, for $\mathbb{P}$-almost all $\omega \in \Omega$, $\varphi(t,\omega)q(\omega)=q(\theta^t\omega)$. (This is precisely the property described in the statement of Lemma~\ref{Dirrfp}.)
\\ \\
We now describe how to ``lift'' the dynamics of $(\Phi^t)_{t \in \mathbb{T}^+}$ from $L^0(\mathbb{P},\mathbb{S}^1)$ onto $\mathcal{M}_1^\mathbb{P}$. For each $t \in \mathbb{T}^+$, define the map $\Theta^t:\Omega \times \mathbb{S}^1 \to \Omega \times \mathbb{S}^1$ by
\[ \Theta^t(\omega,x) \ = \ (\theta^t\omega,\varphi(t,\omega)x). \]

\noindent Note that $(\Theta^t)_{t \in \mathbb{T}^+}$ is a dynamical system on the measurable space $(\Omega \times \mathbb{S}^1,\mathcal{F} \otimes \mathcal{B}(\mathbb{S}^1))$---that is to say: $\Theta^0$ is the identity function; $\Theta^{s+t}=\Theta^t \circ \Theta^s$ for all $s,t \in \mathbb{T}^+$; and $\Theta^t$ is a measurable self-map of $\Omega \times \mathbb{S}^1$ for all $t \in \mathbb{T}^+$.

\begin{lemma} \label{lift}
For any compatible probability measure $\mu$ with disintegration $(\mu_\omega)$ and any $t \in \mathbb{T}^+$, $\Theta^t_\ast\mu$ is a compatible probability measure with disintegration $(\varphi(t,\theta^{-t}\omega)_\ast\mu_{\theta^{-t}\omega})$.
\end{lemma}

\noindent In particular, for any measurable $q:\Omega \to \mathbb{S}^1$ and any $t \in \mathbb{T}^+$, we have that $\Theta^t_\ast\boldsymbol{\delta}_q=\boldsymbol{\delta}_r$ where $r$ is a representative of $\Phi_t(\hat{q})$. Thus the semigroup $(\Theta^t_\ast)_{t \in \mathbb{T}^+}$ on $\mathcal{M}_1^\mathbb{P}$ serves as a natural ``lift'' of the semigroup $(\Phi^t)_{t \in \mathbb{T}^+}$ on $L^0(\mathbb{P},\mathbb{S}^1)$. For a proof of Lemma~\ref{lift}, see [Arn98, Lemma~1.4.4].

\begin{defi}
A probability measure $\mu$ on $(\Omega \times \mathbb{S}^1,\mathcal{F} \otimes \mathcal{B}(\mathbb{S}^1))$ is called an \emph{invariant (probability) measure of $\varphi$} if $\mu$ is compatible and is an invariant measure of the dynamical system $(\Theta^t)_{t \in \mathbb{T}^+}$. In this case, we will refer to any disintegration $(\mu_\omega)$ of $\mu$ as an \emph{invariant random probability measure of $\varphi$}.
\end{defi}

\noindent Note that a random probability measure $(\mu_\omega)$ is invariant if and only if for each $t \in \mathbb{T}^+$, for $\mathbb{P}$-almost all $\omega \in \Omega$, $\varphi(t,\omega)_\ast\mu_\omega=\mu_{\theta^t\omega}$.
\\ \\
Obviously, a measurable function $q:\Omega \to \mathbb{S}^1$ is a random fixed point if and only if $\boldsymbol{\delta}_q$ is an invariant measure. Moreover (as a consequence of Assumption~B) we have the following:

\begin{lemma} \label{rfperg}
For any random fixed point $q:\Omega \to \mathbb{S}^1$, $\boldsymbol{\delta}_q$ is ergodic with respect to $(\Theta^t)_{t \in \mathbb{T}^+}$.
\end{lemma}

\begin{proof}
Fix any $A \in \mathcal{F} \otimes \mathcal{B}(\mathbb{S}^1)$, and for each $t \in \mathbb{T}^+$, let
\[ E_t \ := \ \{ \omega \in \Omega : \Theta^t(\omega,q(\omega)) \in A \}. \]
\noindent Note that for every $t \in \mathbb{T}^+$, $\boldsymbol{\delta}_q(\Theta^{-t}(A) \triangle A)=\mathbb{P}(E_t \triangle E_0)$. Now since $q$ is a random fixed point, we have that for each $t \in \mathbb{T}^+$, $\mathbb{P}(E_t \triangle \theta^{-t}(E_0))=0$.
\\ \\
So, if $\boldsymbol{\delta}_q(\Theta^{-t}(A) \triangle A)=0$ for all $t \in \mathbb{T}^+$, then $\mathbb{P}(\theta^{-t}(E_0) \triangle E_0)=0$ for all $t \in \mathbb{T}^+$, so $\mathbb{P}(E_0) \in \{0,1\}$ (since $\mathbb{P}$ is $(\theta^t)$-ergodic), so $\boldsymbol{\delta}_q(A) \in \{0,1\}$. Thus $\boldsymbol{\delta}_q$ is $(\Theta^t)$-ergodic.
\end{proof}

\noindent We will now describe the set of invariant measures of $\varphi$ when $\mathbb{P}(\Omega_c)=1$.

\begin{thm} \label{simple}
Suppose $\mathbb{P}(\Omega_c)=1$, and let $r:\Omega \to \mathbb{S}^1$ be a measurable function with $r(\omega)=\tilde{r}(\omega)$ for all $\omega \in \Omega_c$. (By Lemma~\ref{crrfp}, $\hat{r}$ is an $\mathcal{F}_+$-measurable random fixed point.) Either:
\begin{enumerate}[\indent (A)]
\item $\boldsymbol{\delta}_r$ is the only $\varphi$-invariant probability measure; or
\item there exists an $\mathcal{F}_-$-measurable random fixed point $a:\Omega \to \mathbb{S}^1$, with $a(\omega) \neq r(\omega)$ for $\mathbb{P}$-almost every $\omega \in \Omega$, such that the set of invariant measures of $\varphi$ is given by $\{\lambda \boldsymbol{\delta}_r + (1-\lambda) \boldsymbol{\delta}_a : \lambda \in [0,1]\}$.
\end{enumerate}
\end{thm}

\noindent Obviously, in case~(B), we have that for $\mathbb{P}$-almost every $\omega \in \Omega_c$, for all $x \in \mathbb{S}^1 \setminus \{r(\omega)\}$, $x \sim_\omega a(\omega)$.

\begin{defi}
If $\mathbb{P}(\Omega_c)=1$ and case~(B) of Theorem~\ref{simple} holds, then we will say that $\varphi$ is \emph{simple}. In this case, letting $r$ and $a$ be as in Theorem~\ref{simple}, we refer to the pair $(\hat{a},\hat{r})$ as the \emph{global random attractor-repeller pair} of $\varphi$.
\end{defi}

\noindent Before proving Theorem~\ref{simple}, it will be useful to introduce the following definition (taken from [KN04]):

\begin{defi}
The \emph{spread} $D(\rho)$ of a probability measure $\rho$ on $\mathbb{S}^1$ is defined as
\[ D(\rho) \ := \ \inf \{ v > 0 \, : \, \exists \, \textrm{closed connected } J \subset \mathbb{S}^1 \textrm{ with } l(J) < v \textrm{ and } \rho(J)>1-v \}. \]
\end{defi}

\noindent It is not hard to show (by considering connected sets with rational endpoints) that $D(\cdot)$ serves as a measurable map from $\mathcal{M}_1$ to $[0,\frac{1}{2}]$, and that $D(\rho)=0$ if and only if $\rho$ is a Dirac mass.

\begin{proof}[Proof of Theorem~\ref{simple}]
Suppose $\boldsymbol{\delta}_r$ is not the only invariant measure, and let $\mu$ be an invariant measure distinct from $\boldsymbol{\delta}_r$. Let $\mu^a$ and $\mu^s$ denote respectively the absolutely continuous and singular parts of the Radon-Nikodym decomposition of $\mu$ with respect to $\boldsymbol{\delta}_r$; note that $\mu^a$ and $\mu^s$ are themselves invariant under the dynamical system $(\Theta^t)_{t \in \mathbb{T}^+}$. By Lemma~\ref{rfperg}, $\boldsymbol{\delta}_r$ is ergodic with respect to $(\Theta^t)$ and therefore $\mu^a$ must be a scalar multiple of $\boldsymbol{\delta}_r$. Hence the probability measure $\nu$ on $\Omega \times \mathbb{S}^1$ given by $\nu(A)=\frac{1}{\mu^s(\Omega \times \mathbb{S}^1)}\mu^s(A)$ is compatible, and is therefore an invariant measure of $\varphi$. Let $(\nu_\omega)_{\omega \in \Omega}$ be a disintegration of $\nu$. By Lemma~\ref{ms}, $\nu_\omega(\{r(\omega)\})=0$ for $\mathbb{P}$-almost all $\omega \in \Omega$. Hence $D(\varphi(t,\omega)_\ast\nu_\omega) \to 0$ as $t \to \infty$ for $\mathbb{P}$-almost every $\omega \in \Omega$. Since $\nu$ is invariant, this implies that for $\mathbb{P}$-almost all $\omega \in \Omega$, $D(\nu_{\theta^n\omega}) \to 0$ as $n \to \infty$ (in the integers). But since $\mathbb{P}$ is $\theta^1$-invariant, it follows that $D(\nu_\omega)=0$ for $\mathbb{P}$-almost all $\omega \in \Omega$, i.e.~$\nu_\omega$ is a Dirac mass for $\mathbb{P}$-almost all $\omega \in \Omega$. So there exists a measurable function $\tilde{a}:\Omega \to \mathbb{S}^1$ such that $\nu_\omega=\delta_{\tilde{a}(\omega)}$ for $\mathbb{P}$-almost all $\omega \in \Omega$ (and so $\nu=\boldsymbol{\delta}_{\tilde{a}}$). Since $\nu$ is $\varphi$-invariant, it follows that $\tilde{a}$ is a random fixed point.
\\ \\
So far, we have seen that any invariant measure $\mu$ is a convex combination of $\boldsymbol{\delta}_r$ and $\boldsymbol{\delta}_{\tilde{a}}$ for some random fixed point $\tilde{a}:\Omega \to \mathbb{S}^1$ such that $\tilde{a}(\omega) \neq r(\omega)$ for $\mathbb{P}$-almost all $\omega \in \Omega$. We next show that up to equivalence, there is \emph{only one} random fixed point that is $\mathbb{P}$-almost everywhere distinct from $r$. Let $a,b:\Omega \to \mathbb{S}^1$ be two random fixed points that are $\mathbb{P}$-almost everywhere distinct from $r$. It is clear that for $\mathbb{P}$-almost every $\omega \in \Omega$, $d(b(\theta^n\omega),a(\theta^n\omega)) \to 0$ as $n \to \infty$ (in the integers). But since $\mathbb{P}$ is $\theta^1$-invariant, it follows that $d(b(\omega),a(\omega))=0$, i.e.~$b(\omega)=a(\omega)$, for $\mathbb{P}$-almost all $\omega \in \Omega$.
\\ \\
So then, the set of all invariant measures takes the form $\{\lambda \boldsymbol{\delta}_r + (1-\lambda) \boldsymbol{\delta}_{\tilde{a}} : \lambda \in [0,1]\}$ for some random fixed point $\tilde{a}:\Omega \to \mathbb{S}^1$ that is $\mathbb{P}$-almost everywhere distinct from $r$. It remains to show that one can modify $\tilde{a}$ on a $\mathbb{P}$-null set to obtain an $\mathcal{F}_-$-measurable function. Fix any point $y \in \mathbb{S}^1$ such that $\mathbb{P}_\ast r(\{y\})=0$. For $\mathbb{P}$-almost every $\omega \in \Omega$, $d(\varphi(n,\omega)y,\tilde{a}(\theta^n\omega)) \to 0$ as $n \to \infty$ (in the integers). Since almost sure convergence implies convergence in probability, it follows that the random variable $\,\omega \mapsto d(\varphi(n,\omega)y,\tilde{a}(\theta^n\omega))\,$ converges in probability to 0 as $n \to \infty$. But since $\mathbb{P}$ is $\theta^n$-invariant for all $n$, this is the same as saying that the random variable $\omega \mapsto \varphi(n,\theta^{-n}\omega)y$ converges in probability to $\tilde{a}$ as $n \to \infty$. Hence in particular, there exists an unbounded increasing sequence $(m_n)_{n \in \mathbb{N}}$ in $\mathbb{N}$ such that for $\mathbb{P}$-almost all $\omega \in \Omega$, $\varphi(m_n,\theta^{-m_n}\omega)y \to \tilde{a}(\omega)$ as $n \to \infty$. So, fixing an arbitrary $k \in \mathbb{S}^1$, the function $a:\Omega \to \mathbb{S}^1$ given by
\[ a(\omega) \ = \ \left\{ \begin{array}{c l} \underset{n \to \infty}{\lim} \; \varphi(m_n,\theta^{-m_n}\omega)y & \textrm{if this limit exists} \\ k & \textrm{otherwise} \end{array} \right. \]
\noindent is $\mathcal{F}_-$-measurable and agrees with $\tilde{a}$ $\mathbb{P}$-almost everywhere. So we are done.
\end{proof}

\noindent Now by Theorem~\ref{cr char}, $\varphi$ is stably synchronising if and only if $\mathbb{P}(\Omega_c)=1$ and $\mathbb{P}_\ast r$ is atomless.

\begin{thm} \label{ssis}
If $\varphi$ is stably synchronising then $\varphi$ is simple.
\end{thm}

\noindent To prove Theorem~\ref{ssis}, we introduce the following:

\begin{defi}
We will say that a compatible probability measure $\mu$ is \emph{past-measurable} if there is a version $(\mu_\omega)$ of the disintegration of $\mu$ such that the map $\omega \mapsto \mu_\omega$ is $(\mathcal{F}_-,\mathfrak{K})$-measurable.
\end{defi}

\begin{lemma} \label{ks}
For every forward-stationary probability measure $\rho$, there exists a past-measurable invariant measure $\mu^\rho$ such that $\rho(A)=\mu^\rho(\Omega \times A)$ for all $A \in \mathcal{B}(\mathbb{S}^1)$.
\end{lemma}

\noindent For a proof (detailing the explicit construction of $\mu^\rho$), see [KS12, Theorem~4.2.9(ii)].

\begin{proof}[Proof of Theorem \ref{ssis}]
Suppose $\mathbb{P}(\Omega_c)=1$ and $\varphi$ is not simple. We know that there exists a forward-stationary probability measure $\rho$. Since $\boldsymbol{\delta}_r$ is the \emph{only} invariant measure of $\varphi$, Lemma~\ref{ks} then implies that $\boldsymbol{\delta}_r$ is past-measurable. So $r$ has an $\mathcal{F}_-$-measurable modification. But $r$ also obviously has an $\mathcal{F}_+$-measurable modification. Since $\mathcal{F}_-$ and $\mathcal{F}_+$ are independent, it follows that $\mathbb{P}_\ast r$ is a Dirac mass, so $\varphi$ is not stably synchronising.
\end{proof}

\section{Contractibility and compressibility}

\begin{defi}
We say that $\varphi$ is \emph{contractible} if for any distinct $x,y \in \mathbb{S}^1$,
\[ \mathbb{P}( \, \omega \, : \, \exists \, t \in \mathbb{T}^+ \textrm{ s.t.~} d(\varphi(t,\omega)x,\varphi(t,\omega)y) < d(x,y) \, ) \ > \ 0. \]
\end{defi}

\begin{lemma} \label{contr}
Suppose $\varphi$ is contractible, and fix any $x,y \in \mathbb{S}^1$. For $\mathbb{P}$-almost all $\omega \in \Omega$ there is an unbounded increasing sequence $(t_n)$ in $\mathbb{T}^+$ such that $d(\varphi(t_n,\omega)x,\varphi(t_n,\omega)y) \to 0$ as $n \to \infty$.
\end{lemma}

\noindent For a proof, see see Section~4.1 of [New17]. (A similar statement can also be found in [BS88, Proposition~4.1].)
\\ \\
Now define the \emph{anticlockwise distance function} $d_+:\mathbb{S}^1 \times \mathbb{S}^1 \to [0,1)$ by
\[ d_+(x,y) \ = \ \min\{ r \geq 0 : \, \pi(x' + r) = y \} \]
\noindent where $x'$ may be any lift of $x$. Obviously $d_+$ is not symmetric, but rather satisfies the relation
\[ d_+(y,x) \ = \ 1 - d_+(x,y). \]
\noindent It is clear that for all $x,y \in \mathbb{S}^1$,
\[ d(x,y) \ = \ \left\{ \! \begin{array}{c l} d_+(x,y) & \textrm{if } d_+(x,y) \leq \frac{1}{2} \\ d_+(y,x) & \textrm{if } d_+(x,y) \geq \frac{1}{2}. \end{array} \right. \]
\noindent Note that $d_+$ is continuous on the set $\{(x,y) \in \mathbb{S}^1 \times \mathbb{S}^1 : x \neq y\}$. For any interval $I \subset \mathbb{R}$ of positive length less than 1, letting $x_1:=\pi(\inf I)$, $x_2:=\pi(\sup I)$ and $J:=\pi(I)$, we have that
\[ l(\varphi(t,\omega)J) \ = \ d_+(\varphi(t,\omega)x_1,\varphi(t,\omega)x_2) \]
\noindent for all $t$ and $\omega$.

\begin{defi}
We say that $\varphi$ is \emph{compressible} if for any distinct $x,y \in \mathbb{S}^1$,
\[ \mathbb{P}( \, \omega \, : \, \exists \, t \in \mathbb{T}^+ \textrm{ s.t.~} d_+(\varphi(t,\omega)x,\varphi(t,\omega)y) < d_+(x,y) \, ) \ > \ 0. \]
\end{defi}

\noindent By reversing the order of inputs, this is equivalent to saying that for any distinct $x,y \in \mathbb{S}^1$,
\[ \mathbb{P}( \, \omega \, : \, \exists \, t \in \mathbb{T}^+ \textrm{ s.t.~} d_+(\varphi(t,\omega)x,\varphi(t,\omega)y) > d_+(x,y) \, ) \ > \ 0. \]

\noindent Perhaps more intuitively, we can also define compressibility in terms of connected subsets of $\mathbb{S}^1$: $\varphi$ is compressible if and only if for every connected set $J \subset \mathbb{S}^1$ with $0<l(J)<1$,
\[ \mathbb{P}( \, \omega \, : \, \exists \, t \in \mathbb{T}^+ \textrm{ s.t.~} l(\varphi(t,\omega)J) < l(J) \, ) \ > \ 0. \]
\noindent Again, this is equivalent to saying that for every connected set $J \subset \mathbb{S}^1$ with $0<l(J)<1$,
\[ \mathbb{P}( \, \omega \, : \, \exists \, t \in \mathbb{T}^+ \textrm{ s.t.~} l(\varphi(t,\omega)J) > l(J) \, ) \ > \ 0. \]

\noindent Obviously, if $\varphi$ is compressible then $\varphi$ is contractible.

\begin{prop}
If $\varphi$ is compressible then for any $x,y \in \mathbb{S}^1$ and $\varepsilon>0$ there exists $t \in \mathbb{T}^+$ such that
\[ \mathbb{P}( \, \omega \, : \, d_+(\varphi(t,\omega)x,\varphi(t,\omega)y) < \varepsilon \, ) \ > \ 0. \]
\end{prop}

\begin{proof}
Suppose we have $x,y \in \mathbb{S}^1$ and $\varepsilon>0$ such that for all $t \in \mathbb{T}^+$,
\[ \mathbb{P}( \, \omega \, : \, d_+(\varphi(t,\omega)x,\varphi(t,\omega)y) < \varepsilon \, ) \ = \ 0. \]
\noindent Let $\Delta$ be the diagonal in $\mathbb{S}^1 \times \mathbb{S}^1$, i.e.~$\Delta=\{(\xi,\xi):\xi \in \mathbb{S}^1\}$. (From now on, we follow the terminology of Section~2.2 of [New17].) Let $G_{(x,y)} \subset \mathbb{S}^1 \times \mathbb{S}^1$ be the smallest closed set containing $(x,y)$ that is forward-invariant under the two-point motion of $\varphi$. The open set $\{(u,v) \in \mathbb{S}^1 \times \mathbb{S}^1 : 0 < d_+(u,v) < \varepsilon\}$ is not accessible from $(x,y)$, and therefore (e.g.~by [New17, Lemma~2.2.3]) $G_{(x,y)}$ does not intersect this set. So if we let $(\bar{u},\bar{v})$ be a point from the compact set $K:=\{ (u,v) \in G_{(x,y)} : \varepsilon \leq d_+(u,v) \leq d_+(x,y) \}$ which minimises $d_+$ on $K$, then $(\bar{u},\bar{v})$ will minimise $d_+$ on the whole of $G_{(x,y)} \setminus \Delta$. Since $G_{(x,y)}$ is forward-invariant, we then have that
\[ \mathbb{P}( \, \omega \, : \, \exists \, t \in \mathbb{T}^+ \textrm{ s.t.~} 0 < d_+(\varphi(t,\omega)\bar{u},\varphi(t,\omega)\bar{v}) < d_+(\bar{u},\bar{v}) \, ) \ = \ 0. \]
\noindent Obviously, since $\varphi(t,\omega)$ is bijective for all $t$ and $\omega$, $d_+(\varphi(t,\omega)\bar{u},\varphi(t,\omega)\bar{v})$ cannot be 0; so we can simply write that
\[ \mathbb{P}( \, \omega \, : \, \exists \, t \in \mathbb{T}^+ \textrm{ s.t.~} d_+(\varphi(t,\omega)\bar{u},\varphi(t,\omega)\bar{v}) < d_+(\bar{u},\bar{v}) \, ) \ = \ 0. \]
\noindent Thus $\varphi$ is not compressible.
\end{proof}

\begin{defi}
We say that $\varphi$ \emph{has reverse-minimal dynamics} if the only open forward-invariant sets are $\emptyset$ and $\mathbb{S}^1$.
\end{defi}

\noindent Obviously if $\varphi$ has a deterministic fixed point $p$ then $\mathbb{S}^1 \setminus \{p\}$ is forward-invariant, and so $\varphi$ does not have reverse-minimal dynamics.

\begin{prop} \label{cts}
If $\mathbb{T}^+=[0,\infty)$ and $\varphi$ is a continuous RDS, then the following are equivalent:
\begin{enumerate}[\indent (i)]
\item $\varphi$ has reverse-minimal dynamics;
\item the only \emph{closed} forward-invariant sets are $\emptyset$ and $\mathbb{S}^1$.
\end{enumerate}
\end{prop}

\begin{rmk}
In general, when $\emptyset$ and $\mathbb{S}^1$ are the only closed forward-invariant sets, we say that $\varphi$ \emph{has minimal dynamics}. So Proposition~\ref{cts} says that for continuous RDS in continuous time, minimal dynamics and reverse-minimal dynamics are equivalent.
\end{rmk}

\begin{proof}[Proof of Proposition \ref{cts}]
We first show that (i)$\Rightarrow$(ii). Suppose we have a closed forward-invariant non-empty proper subset $G$ of $\mathbb{S}^1$; we need to show that there exists an open forward-invariant non-empty proper subset $U$ of $\mathbb{S}^1$. Firstly, if $G$ is a singleton $\{p\}$ then $U:=\mathbb{S}^1 \setminus \{p\}$ is clearly forward-invariant. Now consider the case that $G$ is not a singleton, and let $V$ be a connected component of $\mathbb{S}^1 \setminus G$; we will show that $U:=\mathbb{S}^1 \setminus \bar{V}$ is forward-invariant. (Note that $U$ is non-empty, since $G$ is not a singleton.) Fix any $\omega$ with the property that $\varphi(t,\omega)G \subset G$ for all $t \in \mathbb{T}^+$. Since $\partial V \subset G$, we have that for all $t$, $\varphi(t,\omega)\partial V \subset G$ and therefore in particular $\varphi(t,\omega)\partial V \, \cap \; V = \emptyset$. Now since $\varphi$ is a continuous RDS, it is clear that we can define continuous functions $a,b:[0,\infty) \to \mathbb{R}$ with $a<b$ such that $[a(t),b(t)]$ is a lift of $\varphi(t,\omega)\bar{V}$ for all $t$. (So $\{a(t),b(t)\}$ projects onto $\varphi(t,\omega)\partial V$ for all $t$.) For all $t$, since $\varphi(t,\omega)\partial V \, \cap \; V = \emptyset$, we have that $a(t),b(t) \nin (a(0),b(0))$. Therefore (due to the intermediate value theorem), $a(t) \leq a(0)$ for all $t$ and $b(t) \geq b(0)$ for all $t$. Hence $\bar{V} \subset \varphi(t,\omega)\bar{V}$ for all $t$. Since $\varphi(t,\omega)$ is bijective for all $t$, it follows that $\varphi(t,\omega)U \subset U$ for all $t$. So $U$ is forward-invariant.
\\ \\
Now, in order to show that (ii)$\Rightarrow$(i), first observe that a set $A \subset \mathbb{S}^1$ is forward-invariant if and only if $\mathbb{P}$-almost every $\omega \in \Omega$ has the property that for all $t \in \mathbb{T}^+$,
\[ \varphi(t,\omega)^{-1}(X \setminus A) \ \subset \ X \setminus A. \]
\noindent Hence the fact that (ii)$\Rightarrow$(i) follows from the fact that (i)$\Rightarrow$(ii), except with $\varphi(t,\omega)$ replaced by $\varphi(t,\omega)^{-1}$.
\end{proof}

\begin{prop} \label{comthm}
Suppose $\varphi$ is contractible and has no deterministic fixed point. Then $\varphi$ is compressible if and only if $\varphi$ has reverse-minimal dynamics; and in this case, $\varphi$ is stably synchronising.
\end{prop}

\begin{rmk}
At least in discrete time, if $\varphi$ is contractible and has no deterministic fixed point, then it actually follows automatically that $\varphi$ is stably synchronising, and moreover at an exponential rate; see [Mal14]. (Nonetheless, our proof of stable synchronisation under the additional assumption of compressibility/reverse-minimality is simpler and more elementary than the proof for the results in [Mal14].)
\end{rmk}

\subsection*{Proof of Proposition~\ref{comthm}}

\begin{lemma}
Suppose $\varphi$ is compressible and has no deterministic fixed point. Then $\varphi$ has reverse-minimal dynamics.
\end{lemma}

\begin{proof}
Suppose for a contradiction that $\varphi$ does not have reverse-minimal dynamics, and let $U$ be an open forward-invariant non-empty proper subset of $\mathbb{S}^1$. Let $V$ be a maximal-length connected component of $U$. Since there are no deterministic fixed points, $\mathbb{S}^1 \setminus U$ is not a singleton and so $l(V)<1$. Hence, since $\varphi$ is compressible, there is a positive-measure set of sample points $\omega \in \Omega$ for each of which, for some $t \in \mathbb{T}^+$, $l(\varphi(t,\omega)V)>l(V)$. However, $\varphi(t,\omega)V$ is connected for all $t$ and $\omega$, and so if $l(\varphi(t,\omega)V)>l(V)$ then $\varphi(t,\omega)V$ cannot be a subset of $U$. This contradicts the fact that $U$ is forward-invariant.
\end{proof}

\begin{lemma} \label{diffuse}
Suppose $\varphi$ is contractible and admits a reverse-stationary probability measure $\rho$ that is atomless and has full support. Then $\varphi$ is synchronising.
\end{lemma}

\noindent (We will soon prove that under these same conditions, $\varphi$ is in fact \emph{stably} synchronising.)

\begin{proof}
Fix any distinct $x,y \in \mathbb{S}^1$. Let $J \subset \mathbb{S}^1$ be a connected set with $\partial J = \{x,y\}$. By Lemmas~\ref{mart} and \ref{contr}, there is a $\mathbb{P}$-full set of sample points $\omega$ with the properties that
\begin{enumerate}[\indent (a)]
\item there exists an unbounded increasing sequence $(t_n)$ in $\mathbb{T}^+$ such that
\[ d(\varphi(t_n,\omega)x,\varphi(t_n,\omega)y) \to 0 \hspace{3mm} \textrm{as } n \to \infty \, ; \]
\item $\rho(\varphi(t,\omega)J)$ is convergent as $t \to \infty$.
\end{enumerate}
\noindent Fix any $\omega$ with both these properties, and let $(t_n)$ be as in (a). For any $n$, $d(\varphi(t_n,\omega)x,\varphi(t_n,\omega)y)$ is precisely the smaller of $l(\varphi(t_n,\omega)J)$ and $1 - l(\varphi(t_n,\omega)J)$. Hence there must exist a subsequence $(t_{m_n})$ of $(t_n)$ such that either $l(\varphi(t_{m_n},\omega)J) \to 0$ as $n \to \infty$ or $l(\varphi(t_{m_n},\omega)J) \to 1$ as $n \to \infty$. Since $\rho$ is atomless, it follows that either $\rho(\varphi(t_{m_n},\omega)J) \to 0$ as $n \to \infty$ or $\rho(\varphi(t_{m_n},\omega)J) \to 1$ as $n \to \infty$. Since $\rho(\varphi(t,\omega)J)$ is convergent as $t \to \infty$, it follows that either $\rho(\varphi(t,\omega)J) \to 0$ as $t \to \infty$ or $\rho(\varphi(t,\omega)J) \to 1$ as $t \to \infty$. Since $\rho$ has full support, it follows that either $l(\varphi(t,\omega)J) \to 0$ as $t \to \infty$ or $l(\varphi(t,\omega)J) \to 1$ as $t \to \infty$. Hence $d(\varphi(t,\omega)x,\varphi(t,\omega)y) \to 0$ as $t \to \infty$.
\end{proof}

\begin{lemma} \label{diffusecom}
Under the hypotheses of Lemma~\ref{diffuse}, for any connected $J \subset \mathbb{S}^1$,
\[ \mathbb{P}( \, \omega \, : \, l(\varphi(t,\omega)J) \to 0 \ \mathrm{as} \ t \to \infty \, ) \ = \ 1 - \rho(J). \]
\end{lemma}

\begin{proof}
Fix any connected $J \subset \mathbb{S}^1$. As in the proof of Lemma~3.3, we have that for $\mathbb{P}$-almost every $\omega \in \Omega$, either
\[ \rho(\varphi(t,\omega)J) \to 0 \hspace{3mm} \textrm{and} \hspace{3mm} l(\varphi(t,\omega)J) \to 0  \hspace{3mm} \textrm{as } t \to \infty. \]
\noindent or
\[ \rho(\varphi(t,\omega)J) \to 1 \hspace{3mm} \textrm{and} \hspace{3mm} l(\varphi(t,\omega)J) \to 1  \hspace{3mm} \textrm{as } t \to \infty. \]
\noindent So then, letting $E$ denote the set of sample points $\omega$ for which the latter scenario holds, the dominated convergence theorem gives that as $t \to \infty$,
\[ \int_\Omega \rho(\varphi(t,\omega)J) \, \mathbb{P}(d\omega) \ \to \ \int_\Omega \mathbbm{1}_E(\omega) \, \mathbb{P}(d\omega) \ = \ \mathbb{P}(E). \]
\noindent But we also know that for any $t$,
\[ \int_\Omega \rho(\varphi(t,\omega)J) \, \mathbb{P}(d\omega) \ = \ \rho(J). \]
\noindent Hence $\mathbb{P}(E)=\rho(J)$, i.e.~the probability of the latter scenario is $\rho(J)$ and the probability of the former scenario is $1-\rho(J)$.
\end{proof}

\noindent Combining Lemmas~\ref{diffuse} and \ref{diffusecom}, we have:

\begin{cor} \label{diffusess}
Under the hypotheses of Lemma~\ref{diffuse}, $\varphi$ is stably synchronising.
\end{cor}

\begin{proof}
We already know (from Lemma~\ref{diffuse}) that $\varphi$ is synchronising. Now fix any $x \in X$. Let $(U_n)_{n \in \mathbb{N}}$ be a nested sequence of connected neighbourhoods of $x$ such that $\bigcap_n \hspace{-0.2mm} U_n = \{x\}$. For each $n$,
\begin{align*}
\mathbb{P}( \, \omega \, : \, \exists \, \textrm{open } & U \! \ni x \, \textrm{ s.t.~} l(\varphi(t,\omega)U) \to 0 \textrm{ as } t \to \infty \, ) \\
&\geq \ \mathbb{P}( \, \omega \, : \, l(\varphi(t,\omega)U_n) \to 0 \ \mathrm{as} \ t \to \infty \, ) \\
&= \ 1 - \rho(U_n).
\end{align*}
\noindent But since $\rho$ is atomless, $\rho(U_n) \to 0$ as $n \to \infty$. Hence
\[ \mathbb{P}( \, \omega \, : \, \exists \, \textrm{open } U \! \ni x \, \textrm{ s.t.~} l(\varphi(t,\omega)U) \to 0 \textrm{ as } t \to \infty \, ) \ = \ 1. \]
\noindent So we are done.
\end{proof}

\begin{lemma} \label{open}
Given a dense subset $D$ of $\mathbb{T}^+$, an open set $A \subset \mathbb{S}^1$ is forward-invariant if and only if for each $t \in D$, for $\mathbb{P}$-almost all $\omega \in \Omega$, $\varphi(t,\omega)A \subset A$.
\end{lemma}

\begin{proof}
The ``only if'' direction is clear. Now fix an open set $A \subset \mathbb{S}^1$, and let $D$ be a dense subset of $\mathbb{T}^+$ such that for each $t \in D$, for $\mathbb{P}$-almost all $\omega \in \Omega$, $\varphi(t,\omega)A \subset A$. Let $\tilde{D}$ be a countable dense subset of $D$. $\mathbb{P}$-almost every $\omega \in \Omega$ has the property that for every $t \in \tilde{D}$, $\varphi(t,\omega)A \subset A$ and so $\varphi(t,\omega)^{-1}(\mathbb{S}^1 \setminus A) \subset \mathbb{S}^1 \setminus A$. But since $\mathbb{S}^1 \setminus A$ is closed and the map $t \mapsto \varphi(t,\omega)^{-1}(x)$ is right-continuous for each $x$ and $\omega$, it follows that $\mathbb{P}$-almost every $\omega$ has the property that for every $t \in \mathbb{T}^+$, $\varphi(t,\omega)^{-1}(\mathbb{S}^1 \setminus A) \subset \mathbb{S}^1 \setminus A$ and so $\varphi(t,\omega)A \subset A$. So we are done.
\end{proof}

\begin{lemma} \label{statinv}
For any reverse-stationary probability measure $\rho$, $\mathbb{S}^1 \setminus \mathrm{supp}\,\rho$ is forward-invariant.
\end{lemma}

\begin{proof}
Let $\rho$ be a reverse-stationary probability measure, and let $U:=\mathbb{S}^1 \setminus \mathrm{supp}\,\rho$. For each $t \in \mathbb{T}^+$,
\[ 0 \ = \ \rho(U) \ = \ \int_\Omega \rho(\varphi(t,\omega)U) \, \mathbb{P}(d\omega). \]
\noindent Therefore, for each $t \in \mathbb{T}^+$, for $\mathbb{P}$-almost all $\omega$, $\rho(\varphi(t,\omega)U)=0$ and so $\varphi(t,\omega)U \subset U$. Hence, by Lemma~\ref{open}, $U$ is forward-invariant.
\end{proof}

\begin{lemma} \label{rmdiff}
If $\varphi$ has reverse-minimal dynamics then every reverse-stationary probability measure is atomless and has full support.
\end{lemma}

\begin{proof}
By Lemma~\ref{statinv}, if $\varphi$ has reverse-minimal dynamics then every reverse-stationary probability measure has full support. Now suppose that $\varphi$ has reverse-minimal dynamics and let $\rho$ be a probability measure on $\mathbb{S}^1$ that is not atomless; we will show that $\rho$ is not reverse-stationary. Let $m:=\max\{ \rho(\{x\}) : x \in \mathbb{S}^1 \}$ and let $P:=\{x \in \mathbb{S}^1 : \rho(\{x\})=m \}$. (So $\rho(P)=m|P|$.) Since $\varphi$ has reverse-minimal dynamics, $\mathbb{S}^1 \setminus P$ is not forward-invariant and so (by Lemma~\ref{open}) there must exist $t_0 \in \mathbb{T}^+$ such that
\[ \mathbb{P}( \, \omega \, : \, P \neq \varphi(t_0,\omega)P ) \ > \ 0. \]
\noindent Obviously, for any $\omega$, if $P \neq \varphi(t_0,\omega)P$ then $\rho(\varphi(t_0,\omega)P)<\rho(P)$. Hence we have that
\[ \int_\Omega \rho(\varphi(t_0,\omega)P) \, \mathbb{P}(d\omega) \ < \ \rho(P). \]
\noindent Thus $\rho$ is not reverse-stationary.
\end{proof}

\noindent Combining Lemma~\ref{rmdiff} with Corollary~\ref{diffusess}, we have that if $\varphi$ has reverse-minimal dynamics then $\varphi$ is stably synchronising.

\begin{lemma}
Under the hypotheses of Lemma~\ref{diffuse}, $\varphi$ is compressible.
\end{lemma}

\begin{proof}
For any connected $J \subset \mathbb{S}^1$ with $0 < l(J) < 1$, since $\rho$ has full support, $\rho(J)<1$. Hence, by Lemma~\ref{diffusecom}, there is a positive-measure set of sample points $\omega$ such that $l(\varphi(t,\omega)J) \to 0$ as $t \to \infty$. So in particular, $\varphi$ is compressible.
\end{proof}

\noindent So we are done.

\subsection*{Additive-noise SDE}

Large classes of ordinary differential equations in Euclidean space have been proven to exhibit synchronous behaviour after the addition of Gaussian white noise to the right-hand side (see [CF98] for the one-dimensional case, and [FGS17] for higher-dimensional cases). We shall now do the same for ODEs on $\mathbb{S}^1$. Let $(\Omega,\mathcal{F},\mathbb{P},(\theta^t),(W_t))$ be as in Example~2.1. Let $b:\mathbb{R} \to \mathbb{R}$ be a 1-periodic Lipschitz function, and let $\varphi$ be the RDS on $\mathbb{S}^1$ with trajectories $(\varphi(t,\omega)x)_{t \geq 0}$ whose lifts to $\mathbb{R}$ satisfy the integral equation
\begin{equation} \label{sde} X_t(\omega) \ = \ X_0(\omega) \, + \, \int_0^t b(X_s(\omega)) \hspace{0.1mm} ds \, + \, \sigma \hspace{0.1mm} W_t(\omega) \end{equation}
\noindent (where $\sigma > 0$). It is clear that if $b$ is $\frac{1}{n}$-periodic for some $n > 1$ then $\varphi$ is not synchronising, since any two trajectories starting at a distance $\frac{1}{n}$ apart will remain of distance $\frac{1}{n}$ apart. In the converse direction, Proposition~\ref{comthm} yields the following result:

\begin{prop}
If the least period of $b$ is $1$ then $\varphi$ is stably synchronising.
\end{prop}

\begin{proof}
$\varphi$ clearly has no deterministic fixed points, so we just need to show that $\varphi$ is compressible. Fix a connected $J \subset \mathbb{S}^1$ with $0 < l(J) < 1$. Now, regarding $\Omega$ as being equipped with the topology of uniform convergence on bounded intervals, it is known ([Bur83, Theorem~3.4.1]) that for any $x \in \mathbb{S}^1$ and $t \geq 0$ the map $\omega \mapsto \varphi(t,\omega)x$ is continuous, and it is also known ([Fre13, Proposition~477F]) that the Wiener measure $\mathbb{P}$ assigns strictly positive measure to every non-empty open subset of $\Omega$. In view of these facts, in order to prove compressibility, we just need to find one sample point $\omega_0 \in \Omega$ with the property that at some time $t>0$, $l(\varphi(t,\omega_0)J)<l(J)$. Let $[c_1,c_2] \subset \mathbb{R}$ be a lift of $\bar{J}$ (so $c_2-c_1=l(J)$). Since $b$ is continuous and periodic but \emph{not} $l(J)$-periodic, there must exist $a > c_1$ such that $b(a+l(J))<b(a)$. Now for large $\eta > 0$, consider a sample point $\omega_0^{(\eta)}$ whose path on the time-interval $[0,1]$ is given by
\[ \omega_0^{(\eta)}(t) \ = \ \left\{ \!\! \begin{array}{c l} \eta t & t \in [0,\frac{a - c_1}{\sigma\eta}] \\ \frac{a - c_1}{\sigma} & t \in [\frac{a - c_1}{\sigma\eta},1]. \end{array} \right. \]
\noindent Let $a_1(t)$ be the solution of (\ref{sde}) with $\omega:=\omega_0^{(\eta)}$ and $X_0:=c_1$. Let $a_2(t)$ be the solution of (\ref{sde}) with $\omega:=\omega_0^{(\eta)}$ and $X_0:=c_2$. (So $[a_1(t),a_2(t)]$ is a lift of $\varphi(t,\omega_0^{(\eta)})\bar{J}$ for all $t$.) Provided $\eta$ is sufficiently large: $a_1(\frac{a - c_1}{\sigma\eta})$ will be very close to $a$, and $a_2(\frac{a - c_1}{\sigma\eta})$ will be very close to $a+l(J)$; hence in particular, $b(a_2(\frac{a - c_1}{\sigma\eta}))<b(a_1(\frac{a - c_1}{\sigma\eta}))$, and so there will exist $\delta>0$ such that at time $t:=\frac{a - c_1}{\sigma\eta}+\delta$, we have $l(\varphi(t,\omega_0^{(\eta)})J)=a_2(t)-a_1(t)<l(J)$. Thus $\varphi$ is compressible. Hence Proposition~\ref{comthm} gives that $\varphi$ is stably synchronising.
\end{proof}

\subsubsection*{Acknowledgement}

I would like to thank Dr~Martin Rasmussen, Prof~Jeroen Lamb, Prof~Thomas Kaijser and Maximilian Engel for helpful and interesting discussions. This research was supported by an EPSRC Doctoral Training Account, an EPSRC Doctoral Prize Fellowship, and the DFG grant CRC 701, \emph{Spectral Structures and Topological Methods in Mathematics}.

\subsubsection*{References:}

{[Ant84]} Antonov, V.~A., Modeling of processes of cyclic evolution type. Synchronization by a random signal, \emph{Vestnik Leningradskogo Universiteta Matematika, Mekhanika, Astronomiya} vyp.~2, pp67--76.
\\ \\
{[Arn98]} Arnold, L., \emph{Random Dynamical Systems}, Springer, Berlin Heidelberg New York.
\\ \\
{[Bax86]} Baxendale, P.~H., Asymptotic behaviour of stochastic flows of diffeomorphisms, \emph{Stochastic processes and their applications}, pp1--19.
\\ \\
{[BS88]} Baxendale, P.~H., Stroock, D.~W., Large Deviations and Stochastic Flows of Diffeomorphisms, \emph{Probability Theory and Related Fields} \textbf{80}(2), pp169--215.
\\ \\
{[Bur83]} Burton, T.~A., \emph{Volterra integral and differential equations}, Academic Press, New~York London.
\\ \\
{[CF98]} Crauel, H., Flandoli, F., Additive Noise Destroys a Pitchfork Bifurcation, \emph{Journal of Dynamics and Differential Equations} \textbf{10}(2), pp259--274.
\\ \\
{[Crau02]} Crauel, H., Invariant measures for random dynamical systems on the circle, \emph{Archiv der Mathematik} \textbf{78}(2), pp145--154.
\\ \\
{[Crau02a]} Crauel, H., \emph{Random Probability Measures on Polish Spaces}, Stochastics Monographs 11, Taylor \& Francis, London.
\\ \\
{[FGS17]} Flandoli, F., Gess, B., Scheutzow, M., Synchronization by noise, \emph{Probability Theory and Related Fields} \textbf{168}(3--4), pp511--556.
\\ \\
{[Fre13]} Fremlin, D.~H., \emph{Measure Theory, Volume 4: Topological Measure Spaces} (Chapter~47, version of 8th~April 2013), \url{https://www.essex.ac.uk/maths/people/fremlin/chap47.pdf}.
\\ \\
{[Kai93]} Kaijser, T., On stochastic perturbations of iterations of circle maps, \emph{Physica~D} \textbf{68}(2), pp201--231.
\\ \\
{[Kif86]} Kifer, Y., \emph{Ergodic Theory of Random Transformations}, Birkh\"{a}user, Boston.
\\ \\
{[KN04]} Kleptsyn, V.~A., Nalskii, N.~B., \emph{Contraction of orbits in random dynamical systems on the circle}, Functional Analysis and Its Applications \textbf{38}(4), pp267--282.
\\ \\
{[KS12]} Kuksin, S., Shirikyan, A., \emph{Mathematics of Two-Dimensional Turbulence}, Cambridge Tracts in Mathematics 194, Cambridge University Press, Cambridge~New~York.
\\ \\
{[LeJ87]} Le Jan, Y., \'{E}quilibre statistique pour les produits de diff\'{e}omorphismes al\'{e}atoires ind\'{e}pendants, \emph{Annales de l'Institut Henri Poincar\'{e} Probabilit\'{e}s et Statistiques} \textbf{23}(1), pp111--120.
\\ \\
{[Mal14]} Malicet, D., Random walks on $\mathrm{Homeo}(S^1)$, \url{http://perso.crans.org/mdominique/randomwalks.pdf}.
\\ \\
{[New15]} Newman, J., Ergodic Theory for Semigroups of Markov Kernels (version of 5th~July 2015), \url{http://wwwf.imperial.ac.uk/~jmn07/Ergodic_Theory_for_Semigroups_of_Markov_Kernels.pdf}.
\\ \\
{[New17]} Newman, J., Necessary and sufficient conditions for stable synchronization in random dynamical systems, \emph{Ergodic Theory and Dynamical Systems}, 1-19. doi:10.1017/etds.2016.109.
\\ \\
{[Will91]} Williams, D., \emph{Probability with Martingales}, Cambridge University Press, Cambridge.

\end{document}